\documentclass{article}

\usepackage{microtype}
\usepackage{graphicx}
\usepackage{subfigure}
\usepackage{booktabs} 

\usepackage{hyperref}

\usepackage{amsmath}
\usepackage{amssymb}
\usepackage{mathtools}
\usepackage{amsthm}
\usepackage{xcolor}
\usepackage{algorithm}
\usepackage{algorithmic}
\usepackage[capitalize,noabbrev]{cleveref}

\theoremstyle{plain}
\newtheorem{theorem}{Theorem}[section]
\newtheorem{proposition}[theorem]{Proposition}
\newtheorem{lemma}[theorem]{Lemma}

\theoremstyle{definition}
\newtheorem{definition}[theorem]{Definition}

\theoremstyle{remark}

\newtheorem{example}{Example}

\usepackage[textsize=tiny]{todonotes}

\newcommand{\R}{\mathbb {R}}
\newcommand{\N}{\mathbb {N}}
\newcommand{\Z}{\mathbb {Z}}
\newcommand{\T}{\mathcal {T}}
\newcommand{\U}{\mathcal {U}}

\renewcommand{\L}{\mathcal {L}}

\newcommand{\norm}[1]{\left\lVert#1\right\rVert}
\DeclareMathOperator*{\argmin}{argmin}
\DeclareMathOperator*{\argmax}{argmax}
\DeclareMathOperator*{\E}{\mathbb{E}}
\usepackage{bbold}
\usepackage{paralist}
\usepackage[inline]{enumitem}

\usepackage{authblk}
\title{Principled Acceleration of Iterative Numerical Methods Using Machine Learning}%

\author[1,2]{Sohei Arisaka}
\author[1]{Qianxiao Li}
\affil[1]{Department of Mathematics, National University of Singapore, 119076, Singapore}
\affil[2]{Kajima Corporation, 1078388, Japan}

\begin{document}
\maketitle

\begin{abstract}
    Iterative methods are ubiquitous in large-scale scientific computing applications,
    and a number of approaches based on meta-learning have been recently proposed to accelerate them.
    However, a systematic study of these approaches and how they differ from meta-learning is lacking.
    In this paper, we propose a framework to analyze such learning-based acceleration approaches,
    where one can immediately identify a departure from classical meta-learning.
    We show that this departure may lead to arbitrary deterioration of model performance.
    Based on our analysis, we introduce a novel training method for learning-based acceleration of iterative methods.
    Furthermore, we theoretically prove that the proposed method improves upon the existing methods,
    and demonstrate its significant advantage and versatility through various numerical applications.
\end{abstract}

\section{Introduction}
It is common and important in science and engineering to solve similar computational problems repeatedly.
For example, in an actual project of structural engineering, the structure design is considered iteratively
to satisfy various constraints, such as safety, cost, and building codes.
This iterative design process often involves a large number of structural simulations \cite{Gallet2022-mz}.
Another example is systems biology, where an important but challenging problem is
to estimate parameters of mathematical models for biological systems from observation data,
and solving this inverse problem often requires many numerical simulations \cite{Moles2003-nn, Chou2009-pr}.
In these situations, we can utilize the data of the previously solved problems to solve the next similar but unseen problems more efficiently,
and machine learning is a natural and effective approach for this.

Thus, in recent years, many learning-based methods have been proposed for repeated solutions of computational problems.
These ranges from the direct prediction of solutions as a supervised learning problem
\cite{Guo2016-jx,Tang2017-wc, Shan2020-ps, Ozbay2021-qt, Cheng2021-ua, Pfaff2021-uw, Li2021-bl}
to tightly coupling machine learning and traditional scientific computing to take advantage of both
\cite{Ajuria_Illarramendi2020-tb, Um2020-zx, Huang2020-oa, Vaupel2020-zu, Luna2021-ul, Nikolopoulos2022-ag}.
This paper focuses on the acceleration of iterative algorithms by (meta-)learning
\cite{Feliu-Faba2020-yj, Venkataraman2021-zv,Liu2021-vw,Huang2021-ed,Guo2021-bw, Chen2022-tx,  Psaros2022-ym},
which belongs to the second class.
For instance, \cite{Chen2022-tx} uses meta-learning to generate smoothers of the Multi-grid Network for parametrized PDEs,
and \cite{Guo2021-bw} proposes a meta-learning approach to learn effective solvers based on the Runge-Kutta method for ordinary differential equations.
In \cite{Liu2021-vw, Huang2021-ed, Psaros2022-ym}, meta-learning is used to accelerate the training of physics-informed neural networks for solving PDEs.

However, to date there lacks a systematic scrutiny of the overall approach.
For example, does directly translating a meta-learning algorithm,
such as gradient-based meta-learning \cite{Hospedales2021-ib}, necessarily lead to acceleration in iterative methods?
In this work, we perform a systematic framework to study this problem,
where we find that there is indeed a mismatch between currently proposed training methods
and desirable outcomes for scientific computing.
Using numerical examples and theoretical analysis,
we show that minimizing the solution error does not necessarily speed up the computation.
Based on our analysis, we propose a novel and principled training approach for learning-based acceleration of iterative methods
along with a practical loss function that enables gradient-based learning algorithms.
Our main contributions can be summarized as follows:
\begin{compactenum}
    \item In \cref{sec:current approach}, we propose a general framework, called gradient-based meta-solving, to analyze and develop learning-based numerical methods.
    Using the framework, we reveal the mismatch between the training and testing of learning-based iterative methods in the literature.
    We show numerically and theoretically that the mismatch actually causes a problem.
    \item In \cref{sec:our approach}, we propose a novel training approach to directly minimize the number of iterations
    along with a differentiable loss function for this.
    Furthermore, we theoretically show that our approach can perform arbitrarily better than the current one and numerically confirm the advantage.
    \item In \cref{sec:applications}, we demonstrate the significant performance improvement and versatility of the proposed method through numerical examples,
    including nonlinear differential equations and nonstationary iterative methods.
\end{compactenum}




\section{Problem Formulation and Analysis of Current Approaches}
\label{sec:current approach}
Iterative solvers are powerful tools to solve computational problems.
For example, the Jacobi method and SOR (Successive Over Relaxation) method are used to solve PDEs \cite{Saad2003-vm}.
In iterative methods, an iterative function is iteratively applied to the current approximate solution
to update it closer to the true solution until it reaches a criterion,
such as a certain number of iterations or error tolerance.
These solvers have parameters, such as initial guesses and relaxation factors,
and solver performance is highly affected by the parameters.
However, the appropriate parameters depend on problems and solver configurations,
and in practice, it is difficult to know them before solving problems.

In order to overcome this difficulty, many learning-based iterative solvers have been proposed in the literature
\cite{Hsieh2018-ey, Um2020-zx,Stanziola2021-ah, Chen2022-tx,Kaneda2022-pm, Azulay2022-ee,Nikolopoulos2022-ag}.
However, there lacks a unified perspective to organize and understand them.
Thus, in \cref{sec:general formulation}, we first introduce a general and systematic framework for analysing and developing learning-based numerical methods.
Using this framework, we identify a problem in the current approach
and study how it degrades performance in \cref{sec:min error}.

\subsection{General Formulation of Meta-solving}
\label{sec:general formulation}
Let us first introduce a general framework, called meta-solving, to analyze learning-based numerical methods in a unified way.
We fix the required notations.
A task $\tau$ is any computational problem of interest.
Meta-solving considers the solution of not one but a distribution of tasks, 
so we consider a task space $(\T, P)$ as a probability space that consists of a set of tasks $\T$ and a task distribution $P$.
A loss function $\L:\T \times \U \to \R \cup \{\infty\}$ is a function to measure how well the task $\tau$ is solved,
where $\U$ is the solution space in which we find a solution.
To solve a task $\tau$ means to find an approximate solution $\hat u \in \U$ which minimizes $\L(\tau, \hat u)$.
A solver $\Phi$ is a function from $\T \times \Theta$ to $\U$.
$\theta \in \Theta$ is the parameter of $\Phi$, and $\Theta$ is its parameter space.
Here, $\theta$ may or may not be trainable, depending on the problem.
Then, solving a task $\tau \in \T$ by a solver $\Phi$ with a parameter $\theta$ is denoted by $\Phi(\tau; \theta)= \hat u$.
A meta-solver $\Psi$ is a function from $\T \times \Omega$ to $\Theta$, where $\omega \in \Omega$ is a parameter of $\Psi$ and $\Omega$ is its parameter space.
A meta-solver $\Psi$ parametrized by $\omega \in \Omega$ is expected to generate an appropriate parameter $\theta_\tau \in \Theta$
for solving a task $\tau \in \T$ with a solver $\Phi$, which is denoted by $\Psi(\tau; \omega) = \theta_{\tau}$.
When $\Phi$ and $\Psi$ are clear from the context, we write $\Phi(\tau; \Psi(\tau; \omega))$ as $\hat u (\omega)$
and $\L(\tau, \Phi(\tau; \Psi(\tau; \omega)))$ as $\L(\tau; \omega)$ for simplicity.
Then, by using the notations above, our meta-solving problem is defined as follows:
\begin{definition}[Meta-solving problem]
    For a given task space $(\T, P)$, loss function $\L$, solver $\Phi$, and meta-solver $\Psi$,
    find $\omega \in \Omega$ which minimizes $\E_{\tau \sim P} \left[\L (\tau;\omega) \right]$.
\end{definition}

If $\L$, $\Phi$, and $\Psi$ are differentiable, then gradient-based optimization algorithms, such as SGD and Adam \cite{Kingma2015-ys},
can be used to solve the meta-solving problem.
We call this approach gradient-based meta-solving (GBMS)
as a generalization of gradient-based meta-learning as represented by MAML \cite{Finn2017-qg}.
As a typical example of the meta-solving problem, we consider learning how to choose good initial guesses of iterative solvers
\cite{Ajuria_Illarramendi2020-tb,Vaupel2020-zu,Um2020-zx, Ozbay2021-qt,Luna2021-ul}.
Other examples are presented in \cref{app:examples}.
\begin{example}[Generating initial guesses]
    \label{ex:initial guess}
    Suppose that we need to repeatedly solve similar instances of a class of differential equations with a given iterative solver.
    The iterative solver requires an initial guess for each problem instance, which sensitively affects the accuracy and efficiency of the solution.
    Here, finding a strategy to optimally select an initial guess depending on each problem instance can be viewed as a meta-solving problem.
    For example, let us consider repeatedly solving 1D Poisson equations under different source terms.
    The task $\tau$ is to solve the 1D Poisson equation
    $- \frac{d^2}{dx^2} u(x)  =f(x)$ with Dirichlet boundary condition $u(0) = u(1)=0$.
    By discretizing it with the finite difference scheme,
    we obtain the linear system $Au=f$,
    where $A \in \R^{N\times N}$ and $u,f \in \R^N$. 
    Thus, our task $\tau$ is represented by $\tau = \{f_\tau\}$,
    and the task distribution $P$ is the distribution of $f_\tau$.
    The loss function $\L:\T \times \U \to \R_{\geq 0}$ measures the accuracy of approximate solution $\hat u \in \U = \R^N$.
    For example, $\L(\tau, \hat u) = \norm{A\hat u - f_\tau}^2$ is a possible choice.
    If we have a reference solution $u_\tau$, then $\L(\tau, \hat u) = \norm{\hat u - u_\tau}^2$ can be another choice.
    The solver $\Phi: \T \times \Theta \to \U$ is an iterative solver with an initial guess $\theta \in \Theta$ for the Poisson equation.
    For example, suppose that $\Phi$ is the Jacobi method. 
    The meta-solver $\Psi:\T \times \Omega \to \Theta$ is a strategy characterized by $\omega \in \Omega$
    to select the initial guess $\theta_\tau \in \Theta$ for each task $\tau \in \T$.
    For example, $\Psi$ is a neural network with weight $\omega$.
    Then, finding the strategy to select initial guesses of the iterative solver becomes a meta-solving problem.

\end{example}

Besides scientific computing applications, the meta-solving framework also includes the classical meta-learning problems, such as few-shot learning with MAML,
where $\tau$ is a learning problem, $\Phi$ is one or few steps gradient descent for a neural network starting at initialization $\theta$,
$\Psi$ is a constant function returning its weights $\omega$, and $\omega=\theta$ is optimized to be easily fine-tuned for new tasks.
We remark two key differences between \cref{ex:initial guess} and the example of MAML.
First, in \cref{ex:initial guess}, the initial guess $\theta_\tau$ is selected for each task $\tau$,
while in MAML, the initialization $\theta$ is common to all tasks.
Second, in \cref{ex:initial guess}, the residual $\L(\tau, \hat u) = \norm{A\hat u - f_\tau}^2$
can be used as an oracle to measure the quality of the solution even during testing,
and it tells us when to stop the iterative solver $\Phi$ in practice.
By contrast, in the MAML example, we cannot use such an oracle at the time of testing,
and we usually stop the gradient descent $\Phi$ at the same number of steps as during training.


Here, the second difference gives rise to the question about the choice of $\L$,
which should be an indicator of how well the iterative solver $\Phi$ solves the task $\tau$.
However, if we iterate $\Phi$ until the solution error reaches a given tolerance $\epsilon$,
the error finally becomes $\epsilon$ and cannot be an indicator of solver performance.
Instead, how fast the solver $\Phi$ finds a solution $\hat u$ satisfying the tolerance $\epsilon$ should be used as the performance indicator in this case.
In fact, in most scientific applications, a required tolerance is set in advance,
and iterative methods are assessed by the computational cost, in particular, the number of iterations,
to achieve the tolerance \cite{Saad2003-vm}.
Thus, the principled choice of loss function should be the number of iterations to reach a given tolerance $\epsilon$,
which we refer $\L_\epsilon$ through this paper.

\subsection{Analysis of the Approach of Solution Error Minimization}
\label{sec:min error}
Although $\L_\epsilon$ is the true loss function to be minimized,
minimizing solution errors is the current main approach in the literature for learning parameters of iterative solvers.
For example, in \cite{Um2020-zx}, a neural network is used to generate an initial guess for the Conjugate Gradient (CG) solver.
It is trained to minimize the solution error after a fixed number of CG iterations
but evaluated by the number of CG iterations to reach a given tolerance.
Similarly, \cite{Chen2022-tx} uses a neural network to generate the smoother of PDE-MgNet, a neural network representing the multigrid method.
It is trained to minimize the solution error after one step of PDE-MgNet
but evaluated by the number of steps of PDE-MgNet to reach a given tolerance.

These works can be interpreted that the solution error after $m$ iterations, which we refer $\L_m$, is used as a surrogate of $\L_\epsilon$.
Using the meta-solving framework, it can be understood that
the current approach is trying to train meta-solver $\Psi$ by applying gradient-based learning algorithms to the emprical version of
\begin{equation}
    \label{eq:min res}
    \min_\omega \E_{\tau \sim P}[ \L_m(\tau, \Phi_m(\tau; \Psi(\tau; \omega))) ]
\end{equation}
as a surrogate of
\begin{equation}
    \label{eq:min num}
    \min_\omega \E_{\tau \sim P}[ \L_\epsilon(\tau, \Phi_\epsilon(\tau; \Psi(\tau; \omega))) ],
\end{equation}
where $\Phi_m$ is an iterative solver whose stopping criterion is the maximum number of iterations $m$,
and $\Phi_\epsilon$ is the same kind of iterative solver but has a different stopping criterion, error tolerance $\epsilon$.



The key question is now, is minimizing solution error $\L_m$ sufficient to minimize,
at least approximately, the number of iterations $\L_\epsilon$?
In other words, is (\ref{eq:min res}) a valid surrogate for (\ref{eq:min num})?
We hereafter show that the answer is negative.
In fact, $\L_\epsilon$ can be arbitrarily large even if $\L_m$ is minimized,
especially when the task difficulty (i.e. number of iterations required to achieve a fixed accuracy)
varies widely between instances.
This hightlights a significant departure from classical meta-learning,
and must be taken into account in algorithm design for scientific computing.
Let us first illustrate this phenomena using a numerical example
and its concrete theoretical analysis.
A resolution of this problem leading to our proposed method will be introduced in \cref{sec:our approach}.

\subsubsection{A Counter-example: Poisson Equation}
\label{sec:counter-example}



Let us recall \cref{ex:initial guess} and apply the current approach to it.
The task $\tau = \{f_\tau\}$ is to solve the discretized 1D Poisson equation $Au=f_\tau$.
During training, we use the Jacobi method $\Phi_{m}$,
which starts with the initial guess $\theta = \hat u^{(0)}$ and iterates $m$ times to obtain the approximate solution $\hat u^{(m)} \in \U$.
However, during testing, we use the Jacobi method $\Phi_{\epsilon}$ whose stopping criterion is tolerance $\epsilon$,
which we set $\epsilon = 10^{-6}$ in this example.
We consider two task distributions, $P_1$ and $P_2$, and their mixture $P=pP_1+(1-p)P_2$ with weight $p \in [0, 1]$.
$P_1$ and $P_2$ are designed to generate difficult (i.e. requiring a large number of iterations to solve) tasks and easy ones respectively (\cref{app:counter-example}).
To generate the initial guess $\hat u^{(0)}$, two meta-solvers $\Psi_{\text{nn}}$ and $\Psi_{\text{base}}$ are considered. 
$\Psi_{\text{nn}}$ is a fully-connected neural network with weights $\omega$,
which takes $f_\tau$ as inputs and generates $\hat u_\tau^{(0)}$ depending on each task $\tau$.
$\Psi_{\text{base}}$ is a constant baseline, which gives the constant initial guess $\hat u^{(0)} = \mathbf{0}$ for all tasks.
Note that $\Psi_{\text{base}}$ is already a good choice, because $u_\tau(0)=u_\tau(1)=0$ and $\E_{\tau \sim P}[u_\tau] = \mathbf 0$.
The loss function is the relative error after $m$ iterations, $\L_m = \norm{\hat u^{(m)} - A^{-1}f_\tau}^2/\norm{A^{-1}f_\tau}^2$.
Then, we solve (\ref{eq:min res}) by GBMS as a surrogate of (\ref{eq:min num}).
We note that the case $m=0$, where $\Phi_0$ is the indentity function that returns the initial geuss,
corresponds to using the solution predicted by ordinary supervised learning as the initial guess \cite{Ajuria_Illarramendi2020-tb, Ozbay2021-qt}.
The case $m>0$ is also studied in \cite{Um2020-zx}.
Other details, including the neural network architecture and hyper-parameters for training, are presented in \cref{app:counter-example}.



The trained meta-solvers are assessed by $\L_\epsilon$, the number of iterations to achieve the target tolerance $\epsilon = 10^{-6}$,
which is written by $\L_\epsilon(\tau; \omega) = \min \{m \in \Z_{\geq 0}: \sqrt{L_m(\tau; \omega)} \leq \epsilon \}$ as a function of $\tau$ and $\omega$.
\cref{fig:performance} compares the performance of the meta-solvers trained with different $m$,
and it indicates the failure of the current approach (\ref{eq:min res}).
For the distribution of $p=0$, where all tasks are sampled from the easy distribution $P_2$ and have similar difficulty,
$\Psi_{\text{nn}}$ successfully reduces the number of iterations by 76\% compared to the baseline $\Psi_{\text{base}}$ by choosing a good hyper-parameter $m=25$.
As for the hyper-paramter $m$, we note that larger $m$ does not necessarily lead to better performance,
which will be shown in \cref{thm:opposite},
and the best training iteration number $m$ is difficult to guess in advance.
For the distribution of $p=0.01$,
where tasks are sampled from the difficult distribution $P_1$ with probability $0.01$
and the easy distribution $P_2$ with probability $0.99$,
the performance of $\Psi_{\text{nn}}$ degrades significantly.
In this case, the reduction is only 26\% compared to the baseline $\Psi_{\text{base}}$ even for the tuned hyper-paramter.
\cref{fig:distribution comparison} illustrates the cause of this degradation.
In the case of $p=0$, $\L_m$ takes similar values for all tasks,
and minimizing $\E_{\tau \sim P}[\L_m]$ can reduce the number of iterations for all tasks uniformly.
However, in the case of $p=0.01$, 
$\L_m$ takes large values for the difficult tasks and small values for the easy tasks,
and $\E_{\tau \sim P}[\L_m]$ is dominated by a small number of difficult tasks.
Consequently, the trained meta-solver with the loss $\L_m$ reduces the number of iterations for the difficult tasks
but increases it for easy tasks consisting of the majority (\cref{fig:001_m25}).
In other words, $\L_m$ is sensitive to difficult outliers, which degrades the performance of the meta-solver.

\begin{figure}[hbtp]
    \centering
    \includegraphics[width=0.6\textwidth]{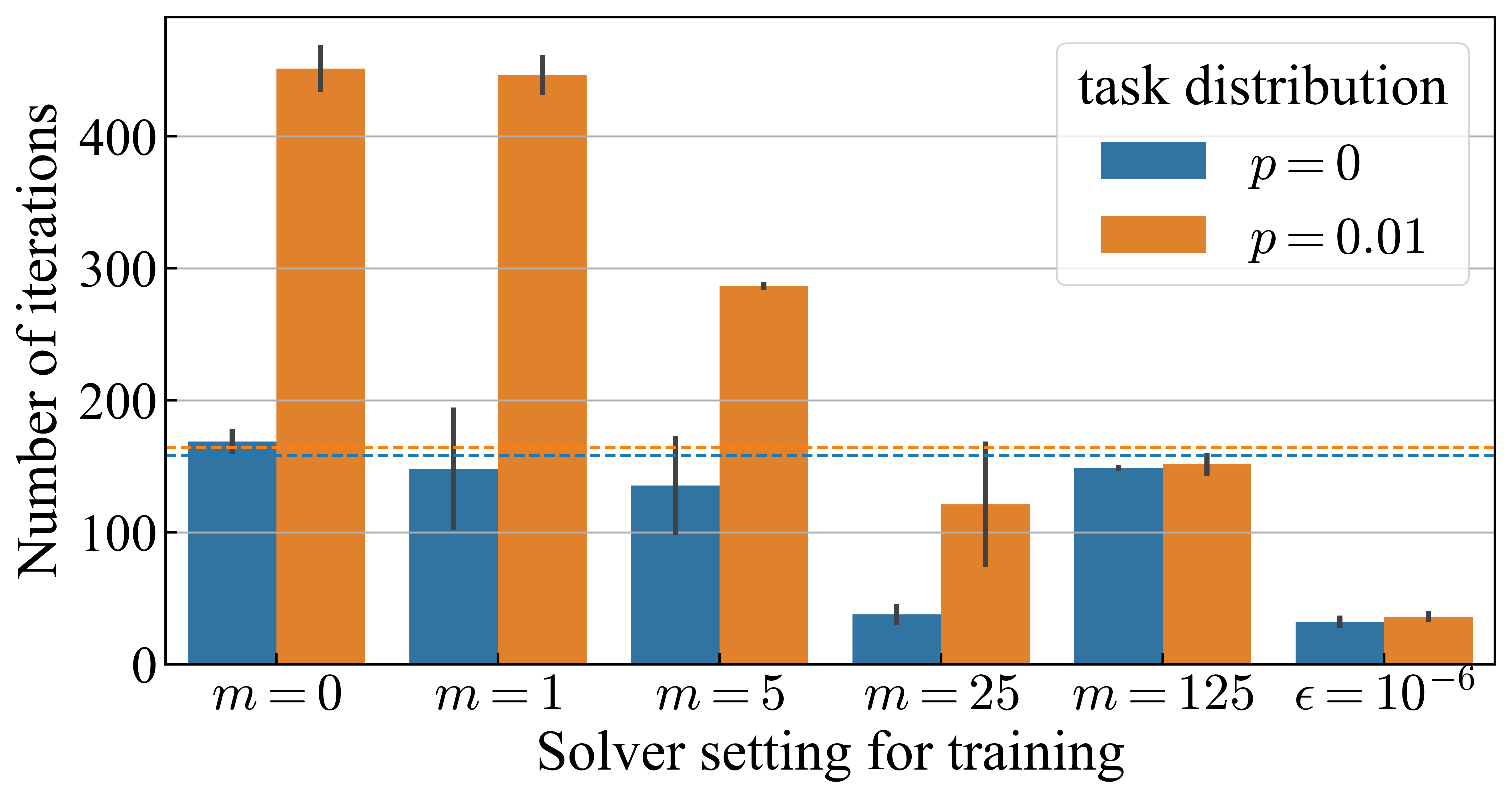} \\ [-2ex]
    \caption{The number of iterations of Jacobi method.
        The two dotted lines indicate the baseline performance of $\Psi_{\text{base}}$ for each task distribution of the corresponding color.
        The error bar indicates the standard deviation.
        Specific numbers are presented in \cref{tab:poisson}.}
    \label{fig:performance}
    \vskip -0.1in
\end{figure}


\begin{figure*}
    \centering
    \subfigure[\footnotesize $p=0$ and $\Psi_{\text{base}}$]{
        \includegraphics[width=0.4\linewidth]{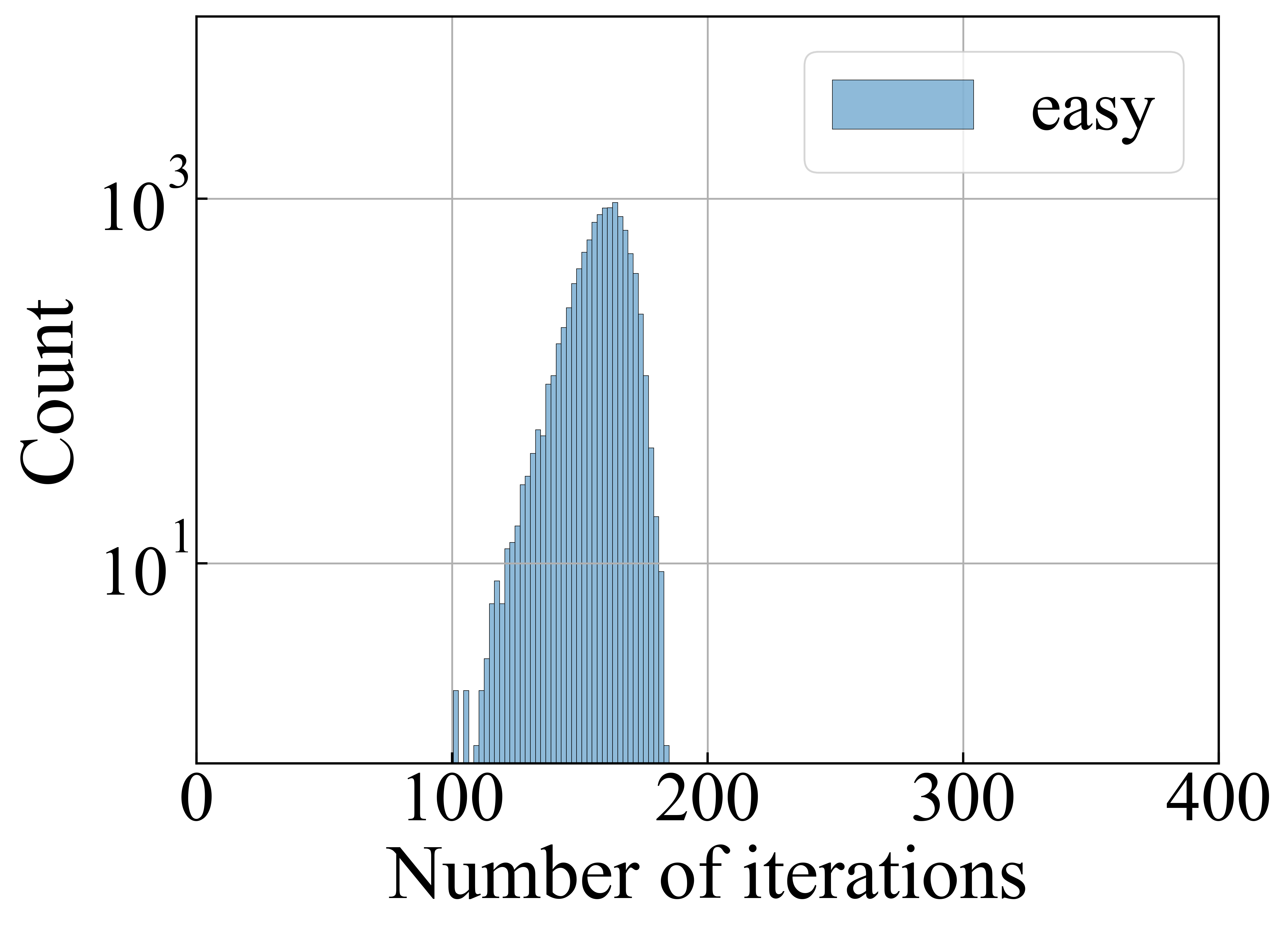}
    }
    \subfigure[\footnotesize $p=0$ and $\Psi_{\text{nn}}$ trained with $m=25$]{
        \includegraphics[width=0.4\linewidth]{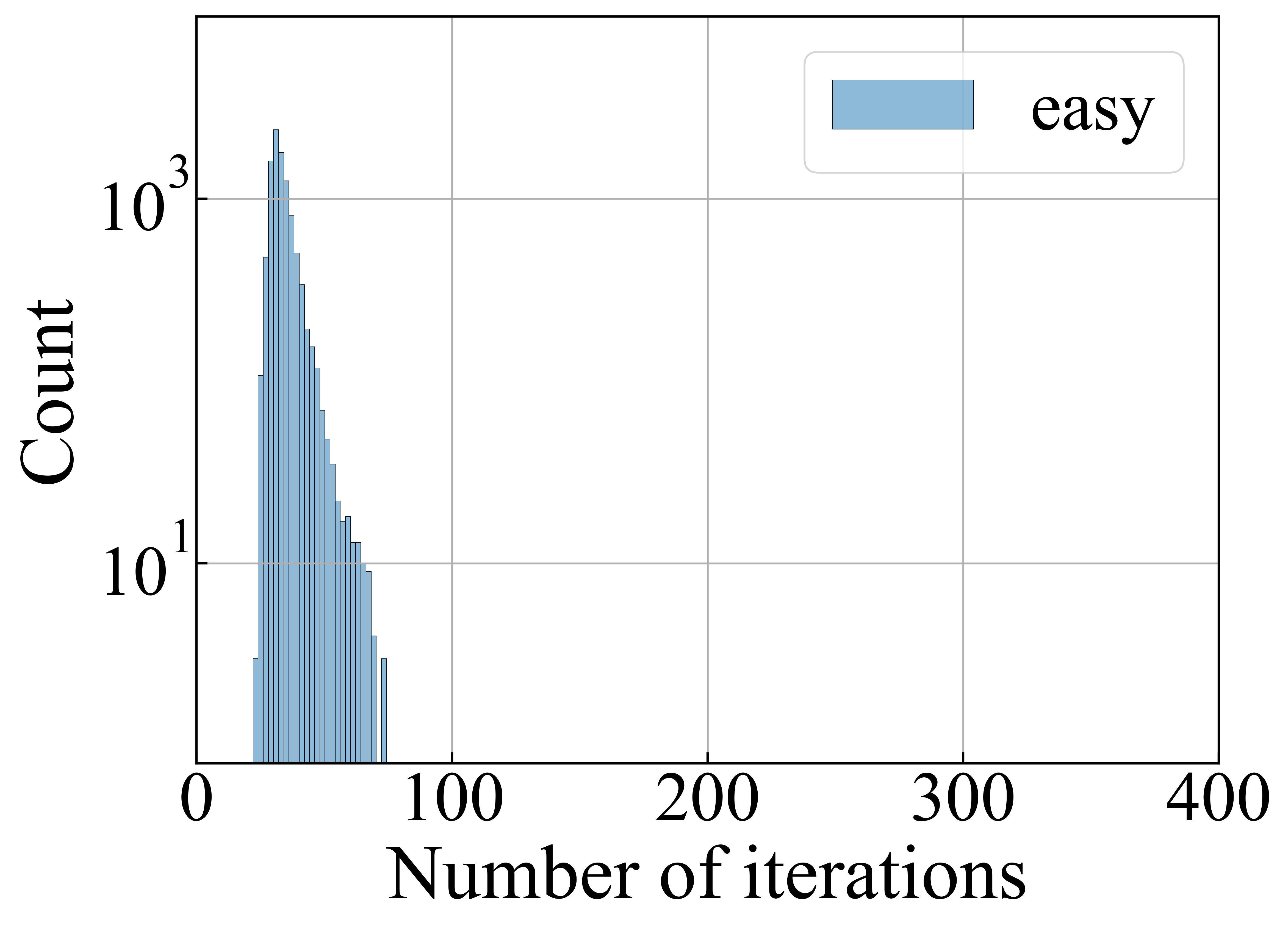}
    } \\
    \subfigure[\footnotesize $p=0$ and $\Psi_{\text{nn}}$ trained with $\epsilon = {10^{-6}}$]{
    \includegraphics[width=0.4\linewidth]{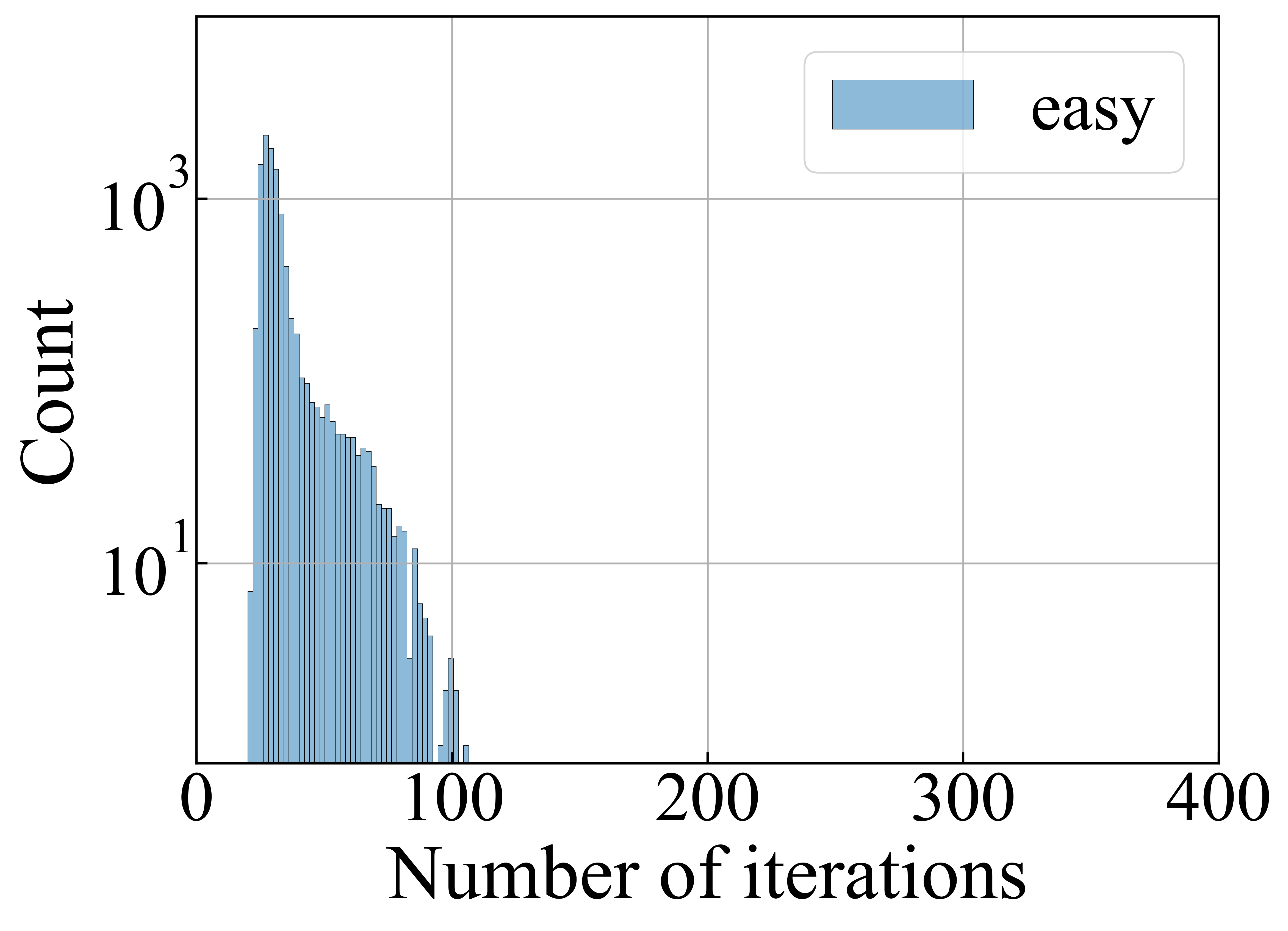}
    }
    \subfigure[\footnotesize $p=0.01$ and $\Psi_{\text{base}}$]{
        \includegraphics[width=0.4\linewidth]{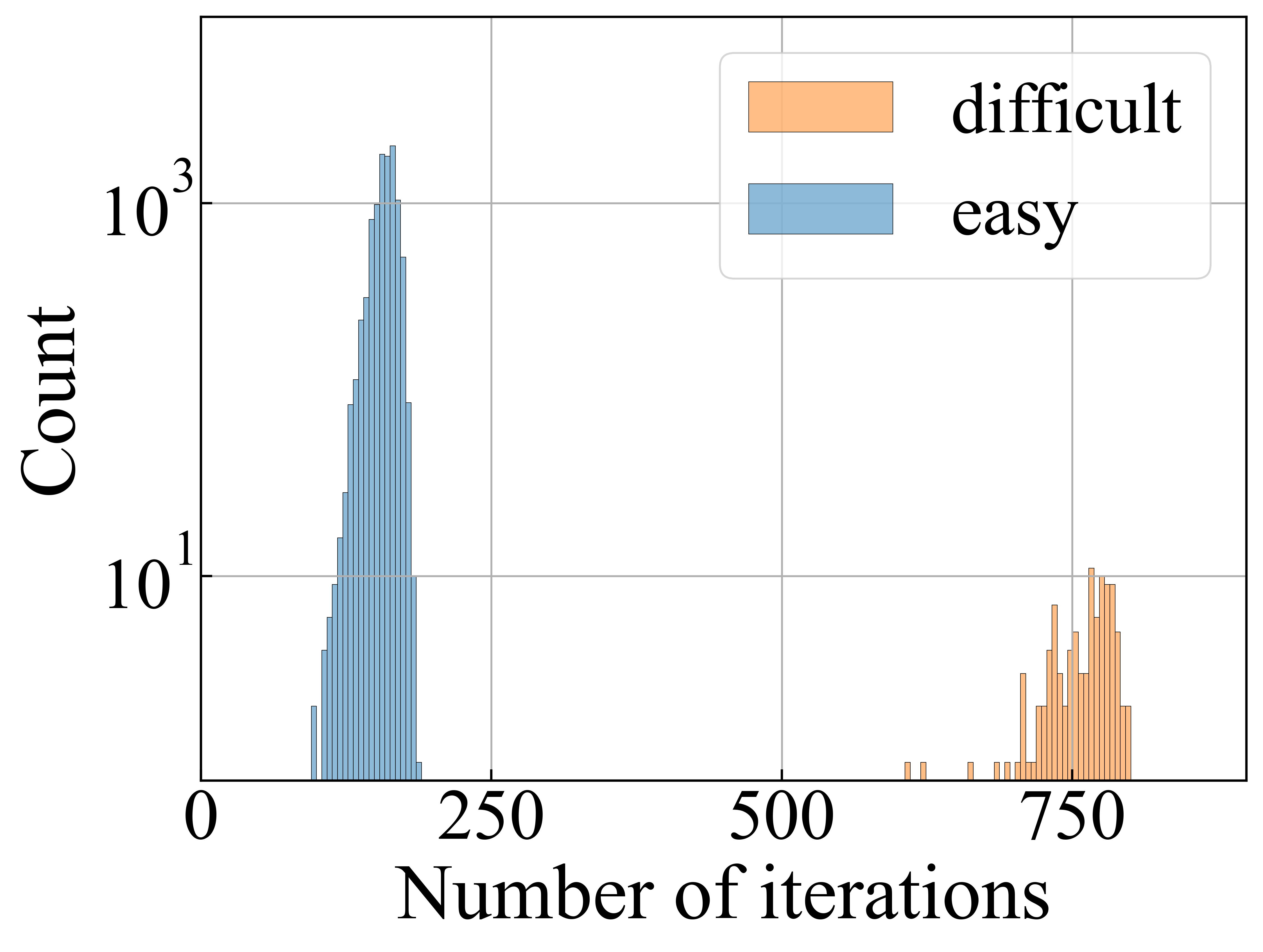}
    }\\
    \subfigure[\footnotesize $p=0.01$ and $\Psi_{\text{nn}}$ trained with $m=25$]{
        \label{fig:001_m25}
        \includegraphics[width=0.4\linewidth]{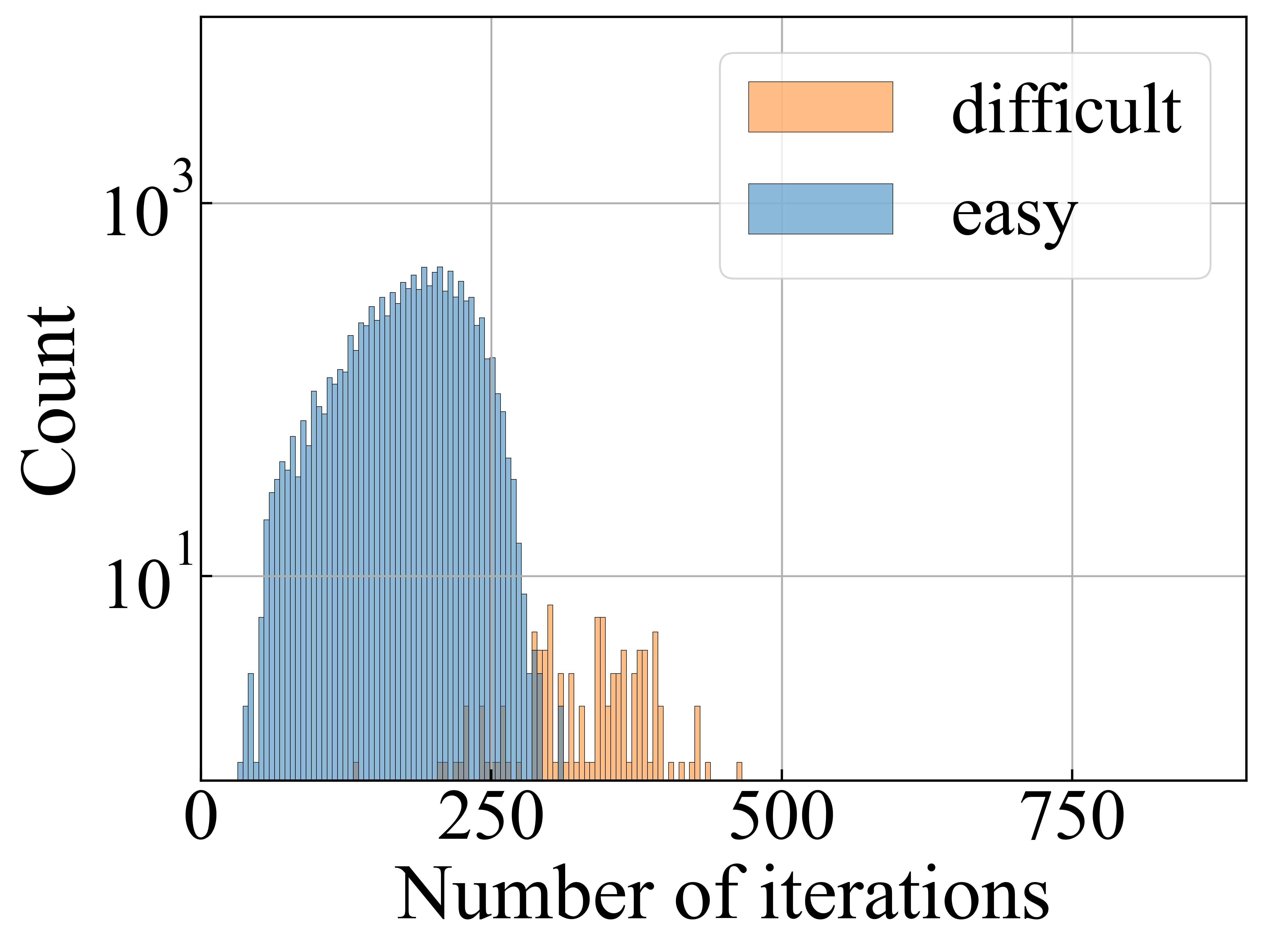}
    }
    \subfigure[\footnotesize $p=0.01$ and $\Psi_{\text{nn}}$ trained with $\epsilon = {10^{-6}}$]{
    \includegraphics[width=0.4\linewidth]{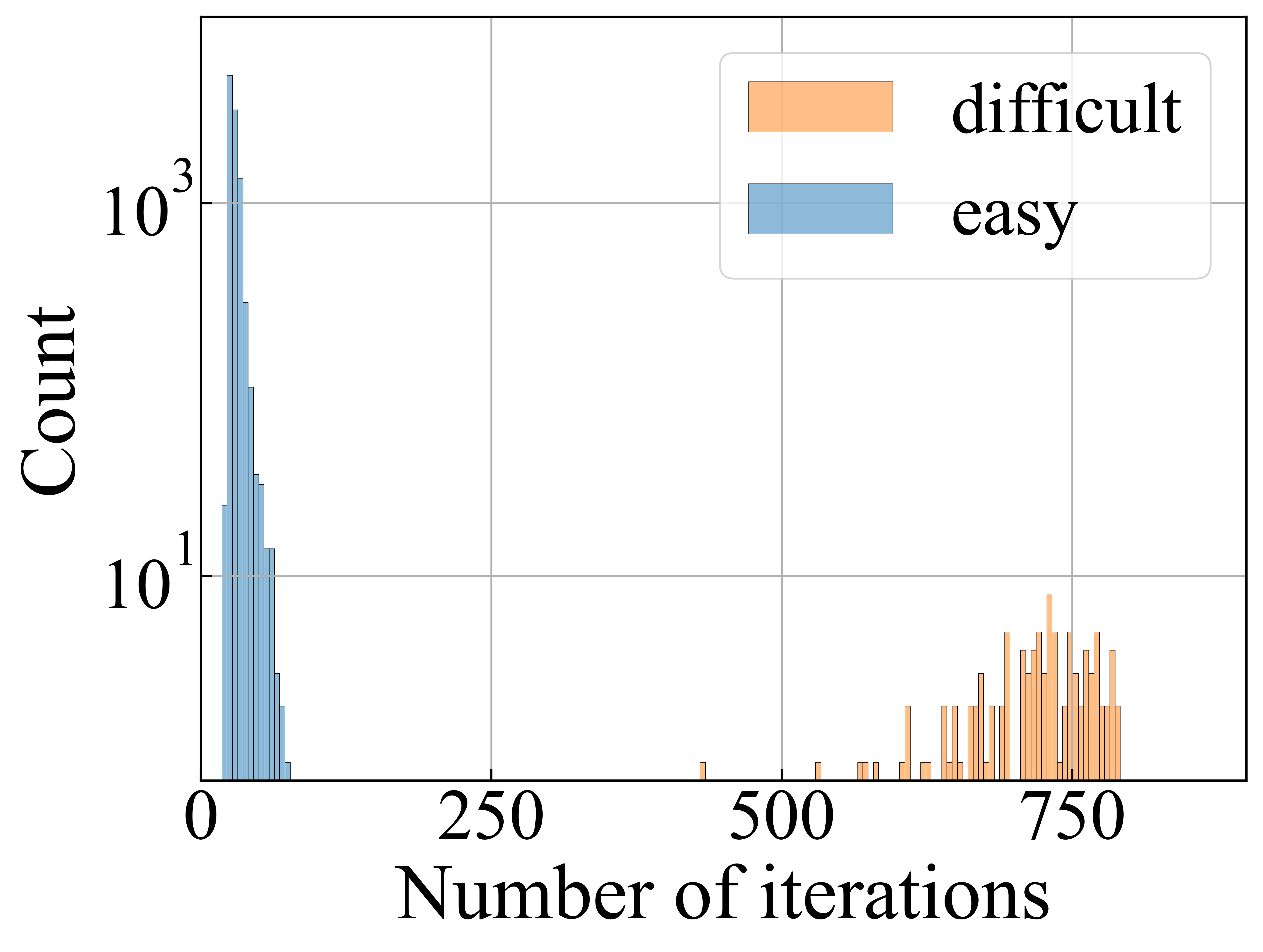}
    }
    \caption{Comparison of distributions of the number of iterations.}
    \label{fig:distribution comparison}
    \vskip -0.2in
\end{figure*}

\subsubsection{Analysis on the Counter-example}
\label{sec:analysis min res}
To understand the property of the current approach (\ref{eq:min res}) and the cause of its performance degradation,
we analyze the counter-example in a simpler setting.
All of the proofs are presented in \cref{app:proofs}.


Let our task space contain two tasks $\T = \{\tau_1, \tau_2\}$,
where $\tau_i = \{f_i\}$ is a task to solve $A u = f_i$.
To solve $\tau_i$, we continue to use the Jacobi method $\Phi_m$ that starts with an initial guess $\hat u^{(0)}$
and gives an approximate solution $\hat u = \hat u^{(m)}$.
We assume that $f_i = c_i \mu_i v_i$, where $c_i \in \R_{\geq 0}$ is a constant, and $v_i$ is the eigenvector of $A$ corresponding to the eigenvalue $\mu_i$.
We also assume $\mu_1 < \mu_2$, which means $\tau_1$ is more difficult than $\tau_2$ for the Jacobi method.
Let the task distribution $P$ generates $\tau_1$ with probability $p$ and $\tau_2$ with probability $1-p$ respectively.
Let $\L_m(\tau_i, \hat u) = \norm{u^{(m)} - A^{-1}f_i}^2$.
We note that the analysis in this paper is also valid for a loss function measuring the residual, e.g. $\L_m(\tau_i, \hat u) = \norm{Au^{(m)} - f_i}^2$.
Suppose that the meta-solver $\Psi$ gives a constant multiple of $f$ as initial guess $\hat u^{(0)}$,
i.e. $\hat u_i ^{(0)} = \Psi(\tau_i; \omega) = \omega f_i$, where $\omega \in \R$ is the trainable parameter of $\Psi$.

For the above problem setting, we can show that the solution error minimization (\ref{eq:min res}) has the unique minimizer
$\omega_m = \argmin_\omega \E_{\tau \sim P}[ \L_m(\tau, \Phi_m(\tau; \Psi(\tau; \omega))) ]$ (\cref{lem:write down}).
As for the minimizer $\omega_m$, we have the following result, which implies a problem in the current approach of solution error minimization.
\begin{proposition}
    \label{thm:opposite}
    For any $p \in (0, 1)$, we have
    \begin{equation}
        \label{eq:limit m}
        \lim_{m \to \infty} \omega_m = \frac{1}{\mu_1}.
    \end{equation}
    For any $\epsilon>0$ and $p \in (0, 1)$, there exists $m_0$ such that for any $m_1$ and $m_2$ satisfying $m_0 < m_1 < m_2$,
    \begin{equation}
        \label{eq:opposite}
        \E_{\tau \sim P}[ \L_\epsilon(\tau; \omega_{m_1}) ]
        \leq \E_{\tau \sim P}[ \L_\epsilon(\tau; \omega_{m_2}) ].
    \end{equation}
    Furthermore, for any $\epsilon > 0$ and $M>0$, there exists task space $(\T, P)$ such that for any $m \geq 0$,
    \begin{equation}
        \label{eq:unbounded}
        \E_{\tau \sim P}[ \L_\epsilon(\tau; \omega_{m}) ] > M.
    \end{equation} 
\end{proposition}
Note that if $\omega = 1/\mu_i$, the meta-solver $\Psi$ gives the best initial guess, the solution, for $\tau_i$,
i.e. $\Psi(\tau_i; 1/\mu_i) = f_i/\mu_i = A^{-1}f_i$.
\cref{thm:opposite} shows the solution error is not a good surrogate loss for learning to accelerate iterative methods in the following sense.
First, \cref{eq:limit m} shows that $\E_{\tau_i \sim P}[\L_m]$ is dominated by the loss for $\tau_1$ as $m$ becomes large.
Hence, $\tau_2$ is ignored regardless of its probability $1-p$, which can cause a problem when $1-p$ is large.
Second, Inequality (\ref{eq:opposite}) implies that training with more iterative steps may lead to worse performance.
Morevoer, Inequality (\ref{eq:unbounded}) shows that the end result can be arbitrarily bad.
These are consistent with numerical results in \cref{sec:counter-example}.
This motivates our subsequent proposal of a principled
approach to achieve acceleration of iterative methods
using meta-learning techniques.

\section{A Novel Approach to Accelerate Iterative Methods}
\label{sec:our approach}

As we observe from the example presented in \cref{sec:current approach},
the main issue is that in general, the solution error $\L_m$ is not a valid surrogate for the number of iterations $\L_\epsilon$.
Hence, the resolution of this problem amounts to finding
a valid surrogate for $\L_\epsilon$ that is also amenable to training.
An advantage of our formalism for GBMS introduced in \cref{sec:general formulation}
is that it tells us exactly the right loss to minimize, which is precisely (\ref{eq:min num}).
It remains then to find an approximation of this loss function,
and this is the basis of the proposed method, which will be explained in \cref{sec:surrogate}.
First, we return to the example in \cref{sec:analysis min res} and show now that minimizing $\L_\epsilon$
instead of $\L_m$ guarantees performance improvement.



\subsection{Minimizing the Number of Iterations}
\label{sec:min num}



\subsubsection{Analysis on the Counter-example}
\label{sec:analysis min num}
To see the property and advantage of directly minimizing $\L_\epsilon$, we analyze the same example in \cref{sec:analysis min res}.
The problem setting is the same as in \cref{sec:analysis min res},
except that we minimize $\L_\epsilon$ instead of $\L_m$ with $\Phi_\epsilon$ instead of $\Phi_m$.
Note that we relax the number of iterations $m \in \Z_{\geq 0}$ to $m \in \R_{\geq 0}$ for the convenience of the analysis.
All of the proofs are presented in \cref{app:proofs}.

As in \cref{sec:analysis min res}, we can find the minimizer
$\omega_\epsilon = \argmin_\omega \E_{\tau \sim P}[ \L_\epsilon(\tau, \Phi_\epsilon(\tau; \Psi(\tau; \omega))) ]$
by simple calculation (\cref{lem:write down2}).
As for the minimizer $\omega_\epsilon$, we have the counterpart of \cref{thm:opposite}:

\begin{proposition} 
    \label{thm:guarantee2}
    we have
    \begin{equation}
        \lim_{\epsilon \to 0}\omega_\epsilon =
        \begin{cases}
            \frac{1}{\mu_1} & \text{if } p > p_0  \\
            \frac{1}{\mu_2} & \text{if } p < p_0,
        \end{cases}
    \end{equation}
    where $p_0$ is a constant depending on $\mu_1$ and $\mu_2$.
    If $\frac{c_1 c_2 (\mu_2 - \mu_1)}{2(c_1 \mu_1 + c_2 \mu_2)} > \delta_1 > \delta_2$ and $p \notin (p_0, p_{\delta_1})$, then for any $\epsilon \leq \delta_2$,
    \begin{equation}
        \label{eq:guarantee2}
        \E_{\tau \sim P}[ \L_\epsilon(\tau; \omega_{\delta_1}) ]
        > \E_{\tau \sim P}[ \L_\epsilon(\tau; \omega_{\delta_2}) ],
    \end{equation}
    where $p_{\delta}$ is a constant depending on tolerance $\delta$ expained in \cref{lem:write down2}, and $p_0$ is its limit as $\delta$ tends to $0$.
    Furthermore, for any $c > 0 $, $\epsilon > 0$, and $\delta \geq \epsilon$, there exists a task space $(\T, P)$ such that for any $m \geq 0$,
    \begin{equation}
        \label{eq:arbitrarily bad}
        \E_{\tau \sim P}[ \L_\epsilon(\tau; \omega_m) ]> c\E_{\tau \sim P}[ \L_\epsilon(\tau; \omega_\delta) ].
    \end{equation}
\end{proposition}
Note that the assumption $\delta_1 < \frac{c_1 c_2 (\mu_2 - \mu_1)}{2(c_1 \mu_1 + c_2 \mu_2)}$
is to avoid the trivial case, $\min_\omega \E_{\tau_i \sim P}[\L_\epsilon(\tau_i; \omega)] = 0$,
and the interval $(p_0, p_{\delta_1})$ is reasonably narrow for such small $\delta_1$.
From \cref{thm:guarantee2}, the following insights can be gleaned.
First, the minimizer $\omega_\epsilon$ has different limits depending on task probability and difficulty,
while the limit of $\omega_m$ is always determined by the difficult task.
In other words, $\L_\epsilon$ and $\L_m$ weight tasks differently.
Second, Inequality (\ref{eq:guarantee2}) guarantees the improvement,
that is, training with the smaller tolerance $\delta_2$ leads to better performance for a target tolerance $\epsilon \leq \delta_2$.
We remark that, as shown in \cref{thm:opposite}, this kind of guarantee is not available in the current approach,
where a larger number of training iterations $m_2$, i.e. larger training cost, does not necessarily improve the performance.
Furthermore, (\ref{eq:arbitrarily bad}) implies the existence of a problem where the current method,
even with hyper-parameter tuning, performs arbitrarily worse than the right problem formulation (\ref{eq:min num}).
These clearly shows the advantage and necessity of a valid surrogate loss function for $\L_\epsilon$,
which we will introduce in \cref{sec:surrogate}.

\subsubsection{Resolving the Counter-example}
\label{sec:counter-example2}
Before introducing our surrogate loss function,
we present numerical evidence that the counter-example, solving Poisson equations in \cref{sec:counter-example},
is resolved by directly minimizing the number of iterations $\L_\epsilon$.
The problem setting is the same as in \cref{sec:counter-example} except for the solver and the loss function.
Instead of solver $\Phi_m$, we use solver $\Phi_\epsilon$ that stops when the error $\L_m$ becomes smaller than $\epsilon$.
Also, we directly minimize $\L_\epsilon$ instead of $\L_m$.
Although $\L_\epsilon$ is discrete and (\ref{eq:min num}) cannot be solved by gradient-based algorithms,
we can use a differentiable surrogate loss function $\tilde \L_\epsilon$ for $\L_\epsilon$, which will be explained in \cref{sec:surrogate}.
The following numerical result is obtained by solving the surrogate problem (\ref{eq:surrogate}) by \cref{alg:minimize Leps}.

\cref{fig:performance} and \cref{fig:distribution comparison} show the advantage of the proposed approach (\ref{eq:min num})
over the current one (\ref{eq:min res}).
For the case of $p=0$, although the current method can perform comparably with ours by choosing a good hyper-parameter $m=25$,
ours performs better without the hyper-parameter tuning.
For the case of $p=0.01$, the current method degrades and performs poorly even with the hyper-parameter tuning,
while ours keeps its performance and reduces the number of iterations
by 78\% compared to the constant baseline $\Psi_{\text{base}}$ and 70\% compared to the best-tuned current method ($m=25$).
We note that this is the case implied in \cref{thm:guarantee2}.
\cref{fig:distribution comparison} illustrates the cause of this performance difference.
While the current method is distracted by a few difficult tasks and increases the number of iterations for the majority of easy tasks,
our method reduces the number of iterations for all tasks in particular for the easy ones.
This difference originates from the property of $\L_m$ and $\L_\delta$ explained in \cref{sec:analysis min num}.
That is, $\L_m$ can be dominated by the small number of difficult tasks
whereas $\L_\delta$ can balance tasks of different difficulties depending on their difficulty and probability.



\subsection{Surrogate Problem}
\label{sec:surrogate}
We have theoretically and numerically shown the advantage of directly minimizing the number of iterations $\L_\epsilon$ over the current approach of minimizing $\L_m$.
Now, the last remaining issue is how to minimize $\E_{\tau \sim P}[\L_\epsilon]$.
As with any machine learning problem, we only have access to finite samples from the task space $(\T, P)$.
Thus, we want to minimize the empirical loss:
$\min_{\omega} \frac{1}{N}\sum_{i=1}^{N}\L_\epsilon(\tau_i; \omega)$.
However, the problem is that $\L_\epsilon$ is not differentiable with respect to $\omega$.
To overcome this issue, we introduce a differentiable surrogate loss function $\tilde \L_\epsilon$ for $\L_\epsilon$.

To design the surrogate loss function, we first express $\L_\epsilon$ by explicitly counting the number of iterations.
We define $\L_\epsilon^{(m)}$ by
\begin{equation}
    \label{eq:indicator}
    \L_\epsilon^{(k+1)}(\tau; \omega) = \L_\epsilon^{(k)}(\tau; \omega) + \mathbb{1}_{\L_k(\tau; \omega) > \epsilon},
\end{equation}
where $\L_\epsilon^{(0)}(\tau; \omega) = 0$,
$\L_k$ is any differentiable loss function that measures the quality of the $k$-th step solution $\hat u^{(k)}$,
and $\mathbb{1}_{\L_k(\tau; \omega) > \epsilon}$
is the indicator function that returns $1$ if $\L_k(\tau; \omega) > \epsilon$ and $0$ otherwise.
Then, we have $\L_\epsilon(\tau; \omega) = \lim_{m\to\infty}\L_\epsilon^{(m)}(\tau; \omega)$ for each $\tau$ and $\omega$.
Here, $\L_\epsilon^{(m)}$ is still not differentiable because of the indicator function.
Thus, we define $\tilde \L_\epsilon^{(m)}$ as a surrogate for $\L_\epsilon^{(m)}$
by replacing the indicator function to the sigmoid function with gain parameter $a$ \cite{Yin2019-kz}.
In summary, $\tilde \L_\epsilon^{(m)}$ is defined by
\begin{equation}
    \tilde \L_\epsilon^{(k+1)}(\tau; \omega) = \tilde \L_\epsilon^{(k)}(\tau; \omega) + \sigma_a(\L_k(\tau; \omega) - \epsilon).
\end{equation}
Then, we have $\L_\epsilon(\tau; \omega) = \lim_{m\to\infty} \lim_{a \to \infty} \tilde \L_\epsilon^{(m)}(\tau; \omega)$
for each $\tau$ and $\omega$.

However, in practice, the meta-solver $\Psi$ may generate a bad solver parameter particularly in the early stage of training.
The bad parameter can make $\L_\epsilon$ very large or infinity at the worst case,
and it can slow down or even stop the training.
To avoid this, we fix the maximum number of iterations $m$ sufficiently large and use $\tilde \L_\epsilon^{(m)}$
as a surrogate for $\L_\epsilon$ along with solver $\Phi_{\epsilon, m}$,
which stops when the error reaches tolerance $\epsilon$ or the number of iterations reaches $m$.
Note that $\Phi_{\epsilon} = \Phi_{\epsilon, \infty}$ and $\Phi_{m} = \Phi_{0, m}$ in this notation.
In summary, a surrogate problem for (\ref{eq:min num}) is defined by
\begin{equation}
    \label{eq:surrogate}
    \min_\omega \E_{\tau \sim P}[ \tilde \L^{(m)}_\epsilon(\tau, \Phi_{\epsilon, m}(\tau; \Psi(\tau; \omega))) ],
\end{equation}
which can be solved by GBMS (\cref{alg:minimize Leps}).

\begin{algorithm}
    \caption{GBMS for minimizing the number of iterations}
    \label{alg:minimize Leps}
    \begin{algorithmic}
        \STATE {\bfseries Input:}
        $(P, \T)$: task space,
        $\Psi$: meta-solver,
        $\Phi$: iterative solver,
        $\phi$: iterative function of $\Phi$,
        $m$: maximum number of iterations of $\Phi$,
        $\epsilon$: tolerance of $\Phi$,
        $\L$: loss function for solution quality,
        $S$: stopping criterion for outer loop,
        $\mathrm{Opt}$: gradient-based algorithm

        \WHILE {$S$ is not satisfied}
        \STATE $\tau \sim P$
        \hfill \COMMENT {sample task $\tau$ from $P$}
        \STATE $\theta_{\tau} \gets \Psi(\tau; \omega)$
        \hfill \COMMENT {generate $\theta_\tau$ by $\Psi$ with $\omega$}
        \STATE $k \gets 0$, $\hat u \gets \Phi_0(\tau ; \theta_\tau)$, $\tilde \L_\epsilon(\tau, \hat u) \gets 0$
        \hfill \COMMENT {initialize}
        \WHILE {$k<m$ or $\L(\tau, \hat u) > \epsilon$}
        \STATE $k \gets k + 1$, $\tilde \L_\epsilon(\tau, \hat u) \gets \tilde \L_\epsilon(\tau, \hat u) + \sigma_a(\L(\tau, \hat u) - \epsilon)$
        \hfill \COMMENT {count iterations and compute surrogate loss}
        \STATE $\hat u \gets \phi(\hat u; \theta_\tau)$
        \hfill \COMMENT {iterate $\phi$ with $\theta_\tau$ to update $\hat u$}

        \ENDWHILE
        \STATE $\omega \gets \mathrm{Opt}(\omega, \nabla_\omega \tilde \L_\epsilon(\tau, \hat u))$
        \hfill \COMMENT {update $\omega$ to minimize $\tilde \L_\epsilon$}
        \ENDWHILE
    \end{algorithmic}
\end{algorithm}

\section{Numerical Examples}

In this section, we show high-performance and versatility of the proposed method (\cref{alg:minimize Leps})
using numerical examples in more complex scenarios,
involving different task spaces, different iterative solvers and their parameters.
Only the main results are presented in this section; details are provided in \cref{app:numerical}.

\label{sec:applications}
\subsection{Generating Relaxation Factors}
\label{sec:relaxation factors}
This section presents the application of the proposed method for generating relaxation factors of the SOR method.
The SOR method is an improved version of the Gauss-Seidel method with relaxation factors enabling faster convergence \cite{Saad2003-vm}.
Its performance is sensitive to the choice of the relaxation factor,
but the optimal choice is only accessible in simple cases (e.g. \cite{Yang2009-dl}) and is difficult to obtain beforehand in general.
Thus, we aim to choose a good relaxation factor to accelerate the convergence of the SOR method using \cref{alg:minimize Leps}.
Moreover, leveraging the generality of \cref{alg:minimize Leps}, we generate the relaxation factor and initial guess simultaneously and observe its synergy.

We consider solving Robertson equation \cite{Robertson1966-ea},
which is a nonlinear ordinary differential equation that models a certain reaction of three chemicals in the form:
\small
\begin{align}
    \label{eq:Robertson}
    \frac{d}{dt}\left(\begin{array}{l}y_{1} \\ y_{2} \\ y_{3}\end{array}\right)
     & =\left(\begin{array}{l}-c_{1} y_{1}+c_{3} y_{2} y_{3} \\ c_{1} y_{1}-c_{2} y_{2}^{2}-c_{3} y_{2} y_{3} \\ c_{2} y_{2}^{2}\end{array}\right), \\
    \left(\begin{array}{l}y_{1}(0) \\ y_{2}(0) \\ y_{3}(0)\end{array}\right)
     & =\left(\begin{array}{l}1 \\ 0 \\ 0 \end{array}\right),  \quad t \in [0, T],
\end{align}
\normalsize
We write $(y_1, y_2, y_3)^T$ as $y$ and the right hand side of \cref{eq:Robertson} as $f(y)$.
Since the Robertson equation is stiff, it is solved by implicit numerical methods \cite{Hairer2010-uw}.
For simplicity, we use the backward Euler method:
$y_{n+1} = y_n + h_n f(y_{n+1})$,
where $h_n$ is the step size.
To find $y_{n+1}$, we need to solve the nonlinear algebraic equation
\begin{equation}
    \label{eq:nonlinear-algebraic}
    g_n(y_{n+1}) := y_{n+1} - h_n f(y_{n+1}) - y_n = 0
\end{equation}
at every time step.
For solving (\ref{eq:nonlinear-algebraic}), we use the one-step Newton-SOR method \cite{Saad2003-vm}.
Our task $\tau$ is to solve (\ref{eq:nonlinear-algebraic}) using the Newton-SOR method,
which is represented by $\tau = \{c_1, c_2, c_3, h_n, y_n \}$.
Note that $\tau$ does not contain the solution of (\ref{eq:nonlinear-algebraic}),
so the training is done in an unsupervised manner.
The loss function $\L_m$ is the residual $\norm{g_n(y_{n+1}^{(m)})}$,
and $\L_\epsilon$ is the number of iterations to have $\L_m < \epsilon$.
As for the task distribution, we sample $c_1, c_2, c_3$ log-uniformly from $[10^{-4}, 1], [10^5, 10^9], [10^2, 10^6]$ respectively.
The solver $\Phi_{\epsilon, m}$ is the Newton-SOR method and its parameter $\theta$ is a pair of relaxation factor and initial guess $(r, y_{n+1}^{(0)})$. 
At test time, we set $m={10}^4$ and $\epsilon=10^{-9}$.
We compare four meta-solvers $\Psi_{\text{base}}$, $\Psi_{\text{ini}}$, $\Psi_{\text{relax}}$, and $\Psi_{\text{both}}$.
The meta-solver $\Psi_{\text{base}}$ has no parameters,
while $\Psi_{\text{ini}}$, $\Psi_{\text{relax}}$, and $\Psi_{\text{both}}$ are implemented by fully-connected neural networks with weight $\omega$.
The meta-solver $\Psi_{\text{base}}$ is a classical baseline,
which uses the previous time step solution $y_n$ as an initial guess $y_{n+1}^{(0)}$ and a constant relaxation factor $r_\text{base}$.
The constant relaxation factor is $r_\text{base}=1.37$, which does not depend on task $\tau$
and is chosen by the brute force search to minimize the average number of iterations
to reach target tolerance $\epsilon = 10 ^ {-9}$ in the whole training data. 
The meta-solver $\Psi_{\text{ini}}$ generates an initial guess $y_\tau$ adaptively,
but uses the constant relaxation factor $r_\text{base}$.
The meta-solver $\Psi_{\text{relax}}$ generates a relaxation factor $r_\tau \in [1, 2]$ adaptively
but uses the previous time step solution as an initial guess.
The meta-solver $\Psi_{\text{both}}$ generates the both adaptively.
Then, the meta-solvers are trained by GBMS with formulations (\ref{eq:min res}) and (\ref{eq:surrogate})
depending on the solver $\Phi_{\epsilon, m}$.
The results are presented in \cref{tab:relax}.
The significant advantage of the proposed approach (\ref{eq:surrogate}) over the baselines and the current approach (\ref{eq:min res})
is observed in \cref{tab:relax different solvers}. 
The best meta-solver $\Psi_{\text{both}}$ with $\Phi_{10^{-9}, 2000}$ reduces the number of iterations by 87\% compared to the baseline,
while all the meta-solvers trained in the current approach (\ref{eq:min res}) fail to outperform the baseline.
This is because the relaxation factor that minimizes the number of iterations for a given tolerance can be
very different from the relaxation factor that minimizes the residual at a given number of iterations.
Furthermore, \cref{tab:relax different meta-solvers} implies that simultaneously generating the initial guess and relaxation factor can create synergy.
This follows from the observation that
the degree of improvement by $\Psi_{\text{both}}$ from $\Psi_{\text{ini}}$ and $\Psi_{\text{relax}}$ (83\% and 29\%)
are greater than that by $\Psi_{\text{relax}}$ and $\Psi_{\text{ini}}$ from $\Psi_{\text{base}}$ (80\% and 20\%).

\begin{table}[hbtp]
    \vskip -0.2in
    {   
    \caption{The number of iterations of the Newton-SOR method.}
    \label{tab:relax}
    \small
    \begin{center}
        \subtable[Different training settings.]{
            \label{tab:relax different solvers}
            \begin{tabular}{lll|rrr}
                $\Psi$                 & $\epsilon$ & $m$  & $\L_\epsilon$ \\
                \hline
                $\Psi_{\mathrm{base}}$ & -          & -    & 70.51         \\
                $\Psi_{\mathrm{both}}$ & -          & 1    & 199.20        \\
                                       & -          & 5    & 105.71        \\
                                       & -          & 25   & 114.43        \\
                                       & -          & 125  & 109.24        \\
                                       & -          & 625  & 190.56        \\
                (ours)                 & $10^{-9}$  & 2000 & \textbf{9.51} \\
            \end{tabular}}  \quad
        \subtable[Different meta-solvers.]{
            \label{tab:relax different meta-solvers}
            \begin{tabular}{l|rrr}
                $\Psi$                  & $\L_\epsilon$ \\
                \hline
                $\Psi_{\mathrm{base}}$  & 70.51         \\
                $\Psi_{\mathrm{ini}}$   & 55.82         \\
                $\Psi_{\mathrm{relax}}$ & 14.38         \\
                $\Psi_{\mathrm{both}}$  & \textbf{9.51} \\
            \end{tabular}}

    \end{center}
    }
    \vskip -0.3in
\end{table}

\subsection{Generating Preconditioning Matrices}
\label{sec:preconditioning}
In this section, \cref{alg:minimize Leps} is applied to preconditioning of the Conjugate Gradient (CG) method.
The CG method is an iterative method that is widely used in solving large sparse linear systems,
especially those with symmetric positive definite matrices.
Its performance is dependent on the condition number of the coefficient matrix,
and a variety of preconditioning methods 
have been proposed to improve it.
In this section, we consider ICCG($\alpha$),
which is the CG method with incomplete Cholesky factorization (ICF) preconditioning with diagonal shift $\alpha$.
In ICCG($\alpha$), the preconditioning matrix $M = \tilde L \tilde L^T$ is obtained by applying ICF to $A + \alpha I$ instead of $A$.
If $\alpha$ is small, $A + \alpha I$ is close to $A$ but the decomposition $\tilde L \tilde L^T \approx A+ \alpha I$ may have poor quality or even fail.
If $\alpha$ is large, the decomposition $\tilde L \tilde L^T \approx A+ \alpha I$ may have good quality but $A + \alpha I$ is far from the original matrix $A$.
Although $\alpha$ can affect the performance of ICCG($\alpha$), there is no well-founded solution to how to choose $\alpha$ for a given $A$ \cite{Saad2003-vm}.
Thus, we train a meta-solver to generate $\alpha$ depending on each problem instance.

We consider linear elasticity equations:
\small
\begin{align}
    -\nabla \cdot \sigma(u) & =f \quad \text { in } B                                     \\
    \sigma(u)               & =\lambda \operatorname{tr}(\epsilon(u)) I+2 \mu \epsilon(u) \\
    \epsilon(u)             & =\frac{1}{2}\left(\nabla u+(\nabla u)^T\right)
\end{align}
\normalsize
where $u$ is the displacement field, $\sigma$ is the stress tensor, $\epsilon$ is strain-rate tensor,
$\lambda$ and $\mu$ are the Lam\'{e} elasticity parameters, and $f$ is the body force.
Specifically, we consider clamped beam deformation under its own weight.
The beam is clamped at the left end, i.e. $u = (0, 0, 0)^T$ at $x=0$, and has length $1$ and a square cross-section of width $W$.
The force $f$ is gravity acting on the beam, i.e. $f = (0, 0, -\rho g)^T$,
where $\rho$ is the density and $g$ is the acceleration due to gravity.
The model is discretized using the finite element method with $8\times 8 \times 2$ cuboids mesh,
and the resulting linear system $Au=f$ has $N=243$ unknowns.
Our task $\tau$ is to solve the linear system $Au=f$,
and it is represented by $\tau = \{ W_\tau, \lambda_\tau, \mu_\tau, \rho_\tau\}$.
Each pysical parameter is sampled from $W_\tau \sim \mathrm{LogUniform}(0.003, 0.3)$, $\lambda_\tau, \mu_\tau, \rho_\tau \sim \mathrm{LogUniform}(0.1, 10)$,
which determine the task space $(\T, P)$.
Our solver $\Phi_{\epsilon, m}$ is the ICCG method and its parameter is diagonal shift $\alpha$.
At test time, we set $m=N=243$, the coefficient matrix size, and $\epsilon=10^{-6}$.
We compare three meta-solvers $\Psi_{\text{base}}, \Psi_{\text{best}}$ and $\Psi_{\text{nn}}$.
The meta-solver $\Psi_{\text{base}}$ is a baseline that provides a constant $\alpha = 0.036$, which is the minimum $\alpha$ that succeeds ICF for all training tasks.
The meta-solver $\Psi_{\text{best}}$ is another baseline that generates the optimal $\alpha_\tau$ for each task $\tau$ by the brute force search.
The meta-solver $\Psi_{\text{nn}}$ is a fully-connected neural network that takes $\tau$ as inputs and outputs $\alpha_\tau$ for each $\tau$.
The loss function $\L_m$ is the relative residual $\L_m = \norm{A\hat u^{(m)} -f}/\norm{f}$,
and the $\L_\epsilon$ is the number of iterations to achieve target tolerance $\epsilon=10^{-6}$.
In addition, we consider a regularization term $\L_\mathrm{ICF}(\tau; \omega) = - \log \alpha_\tau$ to penalize the failure of ICF.
By combining them, we use $\mathbb{1}_{\text{success}}\L_m + \gamma\mathbb{1}_{\text{fail}} \L_{\mathrm{ICF}}$ for training with $\Phi_m$
and $\mathbb{1}_{\text{success}}\tilde \L_\epsilon + \gamma\mathbb{1}_{\text{fail}} \L_{\mathrm{ICF}}$ for training with $\Phi_\epsilon$,
where $\gamma$ is a hyperparameter that controls the penalty for the failure of ICF,
and $\mathbb{1}_{\text{success}}$ and $\mathbb{1}_{\text{fail}}$ indicates the success and failure of ICF respectively.
Then, the meta-solvers are trained by GBMS with formulations (\ref{eq:min res}) and (\ref{eq:surrogate})
depending on the solver $\Phi_{\epsilon, m}$.

\cref{tab:iccg} shows the number of iterations of ICCG with diagonal shift $\alpha$ generated by the trained meta-solvers.
For this problem setting, only one meta-solver trained by the current approach with the best hyper-parameter ($m=125$) outperforms the baseline.
All the other meta-solvers trained by the current approach increase the number of iterations and even fail to converge for the majority of problem instances.
This may be because the error of $m$-th step solution $\L_m$ is less useful for the CG method.
Residuals of the CG method do not decrease monotonically, but often remain high and oscillate for a while, then decrease sharply.
Thus, $\L_m$ is not a good indicator of the performance of the CG method.
This property of the CG method makes it more difficult to learn good $\alpha$ for the current approach.
On the other hand, the meta-solver trained by our approach converges for most instances and reduces the number of iterations
by 55\% compared to $\Psi_{\text{base}}$ and 33\% compared to the best-tuned current approach ($m=125$) without the hyper-parameter tuning.
Furthermore, the performance of the proposed approach is close to the optimal choice $\Psi_{\text{best}}$.

\begin{table}[hbtp]
    \vskip -0.2in
    {   \caption{The number of iterations of ICCG.
        If ICF fails, it is counted as the maximum number of iterations $N=243$.}
    \label{tab:iccg}
    \vskip 0.15in
    \small
    \begin{center}
        \begin{tabular}{lll|rrr}

            $\Psi$                 & $\epsilon$ & $m$ & $\L_\epsilon$  & convergence ratio \\
            \hline
            $\Psi_{\mathrm{base}}$ & -          & -   & 115.26         & 0.864             \\
            $\Psi_{\mathrm{best}}$ & -          & -   & 44.53          & 1.00              \\
            $\Psi_{\mathrm{both}}$ & -          & 1   & 204.93         & 0.394             \\
                                   & -          & 5   & 183.90         & 0.626             \\
                                   & -          & 25  & 193.80         & 0.410             \\
                                   & -          & 125 & 77.08          & \textbf{1.00}     \\
                                   & -          & 243 & 201.65         & 0.420             \\
            (ours)                 & $10^{-6}$  & 243 & \textbf{52.02} & \textbf{0.999}    \\

        \end{tabular}
    \end{center}
    }
    \vskip -0.3in
\end{table}

\section{Conclusion}
In this paper, we proposed a formulation of meta-solving as a general framework to analyze and develop
learning-based iterative numerical methods.
Under the framework, we identify limitations of current approaches directly based on meta-learning,
and make concrete the mechanisms specific to scientific computing resulting in its failure.
This is supported by both numerical and theoretical arguments.
The understanding then leads to a simple but novel training approach
that alleviates this issue.
In particular, we proposed a practical surrogate loss function for directly minimizing the expected
computational overhead, and demonstrated the high-performance and versatility of the proposed method
through a range of numerical examples.
Our analysis highlights the importance of precise formulations and the necessity of modifications
when adapting data-driven workflows to scientific computing.
\nocite{langley00}

\clearpage
\bibliography{example_paper}
\bibliographystyle{unsrt}

\newpage
\appendix
\onecolumn
\section{Organizing related works under meta-solving framework}
\label{app:examples}

\begin{example}[\cite{Feliu-Faba2020-yj}]
    \label{ex:wavelet}
    The authors in \cite{Feliu-Faba2020-yj} propose the neural network architecture with meta-learning approach
    that solves the equations in the form $\mathcal L_{\eta} u(x)=f(x)$ with appropriate boundary conditions,
    where $\mathcal L_\eta$ is a partial differential or integral operator parametrized by a parameter function $\eta(x)$.
    This work can be described in our framework as follows:
    \begin{itemize}
        \item Task $\tau$: The task $\tau$ is to solve a $\mathcal L_{\eta} u(x)=f(x)$ for $\eta = \eta{_\tau}$:
              \begin{itemize}
                  \item Dataset $D_\tau$: The dataset $D_{\tau}$ is $D_{\tau} = \{\eta_{\tau}, f_{\tau}, u_{\tau}\}$,
                        where $\eta_{\tau}, f_{\tau}, u_{\tau} \in \R^N$ are the parameter function, right hand side, and solution respectively.
                  \item Solution space $\mathcal U_\tau$: The solution space $\U_\tau$ is a subset of $\R^N$ for $N \in \N$.
                  \item Loss function $L_\tau$: The loss function $L_\tau:\U \to \R_{\geq 0}$ is the mean squared error with the reference solution,
                        i.e. $L_{\tau}(\hat u) = \norm{u_{\tau} - \hat u}^2$.
              \end{itemize}
        \item Task space $(\T, P)$: The task distribution $(\T, P)$ is determined by the distribution of $\eta_\tau$ and $f_\tau$.
        \item Solver $\Phi$: The solver $\Phi: \T \times \Theta \to \U_\tau$ is implemented by a neural network imitating the wavelet transform,
              which is composed by three modules with weights $\theta = (\theta_1, \theta_2, \theta_3)$.
              In detail, the three modules, $\phi_1(\cdot; \theta_1)$, $\phi_2(\cdot; \theta_2)$, and $\phi_3(\cdot; \theta_3)$,
              represent forward wavelet transform, mapping $\eta$ to coefficients matrix of the wavelet transform, and inverse wavelet transform respectively.
              Then, using the modules, the solver $\Phi$ is represented by $\Phi(\tau;\theta) = \phi_3((\phi_2(\eta_{\tau}; \theta_2) \phi_1(f_{\tau}; \theta_1)); \theta_3) = \hat u$.
        \item Meta-solver $\Psi$: The meta-solver $\Psi:\T \times \Omega \to \Theta$ is the constant function that returns its parameter $\omega$,
              so $\Psi(\tau; \omega) =\omega = \theta$ and $\Omega = \Theta$.
              Note that $\theta$ does not depend on $\tau$ in this example.
    \end{itemize}
\end{example}

\begin{example}[\cite{Psaros2022-ym}]
    \label{ex:PINN}
    In \cite{Psaros2022-ym}, meta-learning is used for learning a loss function of the physics-informed neural network, shortly PINN \cite{Raissi2019-xt}.
    The target equations are the following:
    \begin{align}
        \mathcal{F}_{\lambda}[u](t, x) & =0,(t, x) \in[0, T] \times \mathcal D  \tag{a}\label{eq:a}         \\
        \mathcal{B}_{\lambda}[u](t, x) & =0,(t, x) \in[0, T] \times \partial \mathcal D \tag{b}\label{eq:b} \\
        u(0, x)                        & =u_{0, \lambda}(x), x \in \mathcal D \tag{c}\label{eq:c},
    \end{align}
    where $\mathcal D \subset \mathbb{R}^M$ is a bounded domain, $u:[0, T] \times \mathcal D \to \mathbb{R}^N$ is the solution,
    $\mathcal{F}_{\lambda}$ is a nonlinear operator containing differential operators,
    $\mathcal{B}_{\lambda}$ is a operator representing the boundary condition,
    $u_{0, \lambda}:\mathcal D \to \mathbb{R}^N$ represents the initial condition,
    and $\lambda$ is a parameter of the equations.

    \begin{itemize}
        \item Task $\tau$: The task $\tau$ is to solve a differential equation by PINN:
              \begin{itemize}
                  \item Dataset $D_\tau$: The dataset $D_{\tau}$ is the set of points $(t, x) \in[0, T] \times \mathcal D$ and the values of $u$ at the points if applicable.
                        In detail, $D_{\tau} = D_{f, \tau} \cup D_{b, \tau} \cup D_{u_0, \tau} \cup D_{u, \tau}$, where $D_{f, \tau}, D_{b, \tau}$, and $D_{u_0, \tau}$ are sets of points corresponding to the equation \cref{eq:a}, \cref{eq:b}, and \cref{eq:c} respectively.
                        $D_{u, \tau}$ is the set of points $(t,x)$ and observed values $u(t,x)$ at the points. 
                        In addition, each dataset $D_{\cdot, \tau}$ is divided into training set $D_{\cdot, \tau}^{\text{train}}$ and validation set $D_{\cdot, \tau}^{\text{val}}$.
                  \item Solution space $\mathcal U_\tau$: The solution space $\U_\tau$ is the weights space of PINN.
                  \item Loss function $L_\tau$: The loss function $L_\tau:\U \to \R_{\geq 0}$ is based on the evaluations at the points in $D_{\tau}^{\text{val}}$.
                        In detail,
                        $$L_{\tau}(\hat u) = L_{\tau}^{\text{val}}(\hat u) =L_{f, \tau}^{\text{val}}(\hat u) + L_{b, \tau}^{\text{val}}(\hat u)+ L_{u_0, \tau}^{\text{val}}(\hat u),$$
                        where
                        $$
                            \begin{aligned}
                                {L}_{f, \tau}^{\text{val}}     & =\frac{w_{f}}{|D_{f, \tau}|} \sum_{(t,x) \in D_{f, \tau}} \ell\left(\mathcal{F}_{\lambda}[\hat u](t, x), \mathbf 0\right) \\
                                {L}_{b, \tau}^{\text{val}}     & =\frac{w_{b}}{|D_{b, \tau}|} \sum_{(t,x) \in D_{b, \tau}} \ell\left(\mathcal{B}_{\lambda}[\hat u](t, x), \mathbf 0\right) \\
                                {L}_{u_{0}, \tau}^{\text{val}} & =\frac{w_{u_{0}}}{|D_{u_0, \tau}|} \sum_{(t,x) \in D_{u_0, \tau}} \ell\left(\hat{u}(0, x), u_{0, \lambda}(x)\right),
                            \end{aligned}
                        $$
                        and $\ell: \mathbb{R}^N \times \mathbb{R}^N \to \mathbb{R}_{\geq 0}$ is a function.
                        In the paper, the mean squared error is used as $\ell$.
              \end{itemize}
        \item Task space $(\T, P)$: The task distribution $(\T, P)$ is determined by the distribution of $\lambda$.
        \item Solver $\Phi$: The solver $\Phi: \T \times \Theta \to \U_\tau$ is the gradient descent for training the PINN.
              The parameter $\theta \in \Theta$ controls the objective of the gradient descent,
              $L_{\tau}^{\text{train}}(\hat u; \theta) =L_{f, \tau}^{\text{train}}(\hat u; \theta) + L_{b, \tau}^{\text{train}}(\hat u; \theta)+ L_{u_0, \tau}^{\text{train}}(\hat u; \theta) + L_{u, \tau}^{\text{train}}(\hat u; \theta)$,
              where the difference from $L_{\tau}^{\text{val}}$ is that parametrized loss $\ell_\theta$ is used in $L_{\tau}^{\text{train}}$ instead of the MSE in $L_{\tau}^{\text{val}}$.
              Note that the loss weights $w_f, w_b, w_{u_0}, w_u$ in $L_{\tau}^{\text{train}}$ are also considered as part of the parameter $\theta$.
              In the paper, two designs of $\ell_\theta$ are studied.
              One is using a neural network, and the other is using a learned adaptive loss function.
              In the former design, $\theta$ is the weights of the neural network, and in the latter design, $\theta$ is the parameter in the adaptive loss function.
        \item Meta-solver $\Psi$: The meta-solver $\Psi:\T \times \Omega \to \Theta$ is the constant function that returns its parameter $\omega$, so $\Psi(\tau; \omega) =\omega = \theta$ and $\Omega = \Theta$.
              Note that $\theta$ does not depend on $\tau$ in this example.
    \end{itemize}
\end{example}

\begin{example}[\cite{Yonetani2021-ez}]
    \label{ex:A*}
    In \cite{Yonetani2021-ez}, the authors propose a learning-based search method, called Neural A*, for path plannning.
    This work can be described in our framework as follows:
    \begin{itemize}
        \item Task $\tau$: The task $\tau$ is to solve a point-to-point shortest path problem on a graph $G=(V, E)$.
              \begin{itemize}
                  \item Dataset $D_\tau$: The dataset $D_{\tau}$ is $\{G_\tau, v_{\tau, s}, v_{\tau, g}, p_\tau\}$,
                        where $G_\tau = (V_\tau, E_\tau)$ is a graph (i.e. environmental map of the task), $v_{\tau, s} \in V_\tau$ is a starting point,
                        $v_{\tau, g} \in V_\tau$ is a goal, and $p_\tau$ is the ground truth path from $v_{\tau, s}$ to $v_{\tau, g}$.
                  \item Solution space $\mathcal U_\tau$: The solution space $\U_\tau$ is $\{0, 1\}^{V_\tau}$.
                  \item Loss function $L_\tau$: The loss function $L_\tau$ is $L_{\tau}(\hat p) = \norm{\hat p - p}_1 / |V_\tau|$,
                        where $\hat p$ is the search history by the A* algorithm.
                        Note that the loss function considers not only the final solution but also the search history
                        to improve the efficiency of the node explorations.
              \end{itemize}
        \item Task space $(\T, P)$: The task distribution $(\T, P)$ is determined according to each problem.
              The paper studies both synthetic and real world datasets.
        \item Solver $\Phi$: The solver $\Phi: \T \times \Theta \to \U_\tau$ is the differentiable A* algorithm proposed by the authors,
              which takes $D_{\tau} \setminus \{p_\tau\}$ and the guidance map $\theta_\tau \in \Theta = [0, 1]^{V_\tau}$
              that imposes a cost to each node $v \in V_\tau$,
              and returns the search history $\hat p$ containing the solution path.
        \item Meta-solver $\Psi$: The meta-solver $\Psi:\T \times \Omega \to \Theta$ is 2D U-Net with weights $\omega \in \Omega$,
              which takes $D_{\tau} \setminus \{p_\tau\}$ as its input and returns the guidance map $\theta_\tau$ for the A* algorithm $\Phi$.
    \end{itemize}

\end{example}

\begin{example}[\cite{Cheng2021-ua}]
    \label{ex:mg}
    The target equation in \cite{Cheng2021-ua} is a linear systems of equations $A_\eta u = f$ obtained by discretizing parameterized steady-state PDEs, where $u, f \in \mathbb{R}^N$ and $A_\eta \in \mathbb{R}^{N\times N}$ is determined by $\eta$, a parameter of the original equation.
    This work can be described in our framework as follows:
    \begin{itemize}
        \item Task $\tau$: The task $\tau$ is to solve a linear system $A_\eta u = f$ for $\eta_ = \eta_\tau$:
              \begin{itemize}
                  \item Dataset $D_\tau$: The dataset $D_{\tau}$ is $\{ \eta_{\tau}, f_{\tau}\}$.
                  \item Solution space $\mathcal U_\tau$: The solution space $\U_\tau$ is $\R^N$.
                  \item Loss function $L_\tau$: The loss function $L_\tau:\U \to \R_{\geq 0}$ is an unsupervised loss based on the residual of the equation,
                        $L_{\tau}(\hat u) = {\norm{f_{\tau} - A_{\eta_{\tau}}\hat u}^2} / {\norm{f_{\tau}}^2}$.
              \end{itemize}
        \item Task space $(\T, P)$: The task distribution $(\T, P)$ is determined by the distribution of $\eta_\tau$ and $f_\tau$.
        \item Solver $\Phi$: The solver $\Phi: \T \times \Theta \to \U$ is iterations of a function $\phi_{\tau}(\cdot;\theta):\mathcal U \to \mathcal U$ that represents an update step of the multigrid method.
              $\phi_{\tau}$ is implemented using a convolutional neural network and its parameter $\theta$ is the weights corresponding to the smoother of the multigrid method.
              Note that weights of $\phi_{\tau}$ other than $\theta$ are naturally determined by $\eta_{\tau}$ and the discretization scheme.
              In addition, $\phi_{\tau}$ takes $f_{\tau}$ as part of its input at every step, but we write these dependencies as $\phi_{\tau}$ for simplicity.
              To summarize, $\Phi(\tau; \theta) = \phi_{\tau}^k(u^{(0)}; \theta) = \hat u$, where $k$ is the number of iterations of the multigrid method and $u^{(0)}$ is initial guess, which is $\mathbf 0$ in the paper.
        \item Meta-solver $\Psi$: The meta-solver $\Psi:\T \times \Omega \to \Theta$ is implemented by a neural network with weights $\omega$, which takes $A_{\eta_{\tau}}$ as its input and returns weights $\theta_{\tau}$ that is used for the smoother inspired by the subspace correction method.
    \end{itemize}
\end{example}

\begin{example}[\cite{Um2020-zx}]
    \label{ex:SIL}
    In \cite{Um2020-zx}, initial guesses for the Conjugate Gradient (CG) solver are generated by using a neural network.
    This work can be described in our framework as follows:
    \begin{itemize}
        \item Task $\tau$: The task $\tau$ is to solve a pressure Poisson equation $\nabla \cdot \nabla p=\nabla \cdot v$ on 2D,
              where $p \in \R^{d_x \times d_y}$ is a pressure field and $v \in \R^{2\times d_x \times d_y}$ is a velocity filed:
              \begin{itemize}
                  \item Dataset $D_\tau$: The dataset $D_{\tau}$ is $D_{\tau} = \{v_\tau \}$,
                        where $v_\tau$ is a given velocity sample.
                  \item Solution space $\mathcal U_\tau$: The solution space $\U_\tau$ is $\R^{d_x \times d_y}$.
                  \item Loss function $L_\tau$: The loss function $L_\tau$ is $L_{\tau}(\hat p) = \norm{\hat p - p^{(0)}}^2$,
                        where $\hat p$ is the approximate solution of the Poisson equation by the CG solver, and $p^{(0)}$ is the initial guess.
              \end{itemize}
        \item Task space $(\T, P)$: The task distribution $(\T, P)$ is determined by the distribution of $v_\tau$.
        \item Solver $\Phi$: The solver $\Phi_k: \T \times \Theta \to \U_\tau$ is the differentiable CG solver with $k$ iterations.
              Its parameter $\theta \in \Theta$ is the initial guess $p^{(0)}$, so $\Phi_k(\tau; \theta) = \hat p$.
        \item Meta-solver $\Psi$: The meta-solver $\Psi:\T \times \Omega \to \Theta$ is 2D U-Net with weights $\omega \in \Omega$,
              which takes $v_\tau$ as its input and returns the initigal guess $p^{(0)}_{\tau}$ for the CG solver.
    \end{itemize}

\end{example}

\section{Proofs in \texorpdfstring{\cref{sec:current approach}}{}}
\label{app:proofs}
We first present a lemma about the eigenvalues and eigenvectors of $A$ and the corresponding Jacobi iteration matrix $M=I-\frac{1}{2}A$.

\begin{lemma}
    \label{lem:eigen}
    The eigenvalues of $A$ and $M$ are $\mu_i = 2 - 2 \cos \frac{i}{N+1}\pi$ and $\lambda_i = \cos \frac{i}{N+1}\pi$ for $i=1,2,\ldots,N$ respectively.
    Their common corresponding eigenvectors are $v_i = (\sin\frac{1}{N+1}i\pi, \sin\frac{2}{N+1}i\pi, \ldots, \sin\frac{N}{N+1}i\pi)^T$.
\end{lemma}
\begin{proof}[Proof of \cref{lem:eigen}]
    The proof is presented in Section 9.1.1 of \cite{Greenbaum1997-mv}.
\end{proof}

We can write down the loss functions, $\L_m$ and $\L_\epsilon$, and find their minimiziers.
\begin{proposition} [Minimizer of (\ref{eq:min res})]
    \label{lem:write down}
    (\ref{eq:min res}) is written as
    \begin{equation}
        \label{eq:res_exp}
        \min_\omega\ p c_1^2\left(\omega \mu_1-1\right)^2 \lambda_1^{2 m}
        +(1-p) c_2^2\left(\omega \mu_2-1\right)^2 \lambda_2^{2 m},
    \end{equation}
    and it has the unique minimizer
    \begin{equation}
        \omega_m = \frac{p c_1^2 \mu_1 \lambda_1^{2 m}+(1-p) c_2^2 \mu_2 \lambda_2^{2 m}}{p c_1^2 \mu_1^2 \lambda_1^{2 m}+(1-p) c_2^2 \mu_2^2 \lambda_2^{2 m}}.
    \end{equation}
    Furthermore, for any $p \in (0, 1)$, we have
    \begin{equation}
        \label{eq:res_limit}
        \lim_{m\to \infty}\omega_m = \frac{1}{\mu_1}.
    \end{equation}
\end{proposition}
\begin{proof}[Proof of \cref{lem:write down}]
    We have
    \begin{align}
        \E_{\tau_i \sim P}[ \L_m(\tau_i, ;\omega) ]
         & = p\L(\tau_1; \omega) + (1-p)\L(\tau_2; \omega)                                                       \\
         & = p\norm{\hat u_1^{(m)}(\omega) - u_1^*}^2 + (1-p)\norm{\hat u_2^{(m)}(\omega) - u_2^*}^2             \\
         & = p\norm{M^m (\hat u_1^{(0)}(\omega) - u_1^*)}^2 + (1-p)\norm{M^m (\hat u_2^{(0)}(\omega) - u_2^*)}^2 \\
         & = p\norm{M^m (\omega c_1\mu_1 v_1 - c_1 v_1)}^2 + (1-p)\norm{M^m (\omega c_2\mu_2 v_2 - c_2 v_2)}^2   \\
         & = p \lambda_1^{2m}c_1^2(\omega \mu_1 - 1)^2 + (1-p) \lambda_2^{2m}c_2^2(\omega \mu_2 - 1)^2.          \\
    \end{align}
    Since $\E_{\tau_i \sim P}[ \L_m(\tau_i, ;\omega) ]$ is a quadratic function of $\omega$,
    its minimum is achieved at
    \begin{equation}
        \omega_m = \frac{p c_1^2 \mu_1 \lambda_1^{2 m}+(1-p) c_2^2 \mu_2 \lambda_2^{2 m}}{p c_1^2 \mu_1^2 \lambda_1^{2 m}+(1-p) c_2^2 \mu_2^2 \lambda_2^{2 m}}.
    \end{equation}
    Since $\lambda_1 > \lambda_2$, its limit is
    \begin{equation}
        \lim_{m\to \infty}\omega_m =\frac{p c_1^2 \mu_1 }{p c_1^2 \mu_1^2 }  = \frac{1}{\mu_1}.
    \end{equation}

\end{proof}

\begin{proposition} [Minimizer of (\ref{eq:min num})]
    \label{lem:write down2}
    (\ref{eq:min num}) is written as
    \begin{equation}
        \label{eq:num_exp}
        \min_\omega\ p \left(\frac{\log \frac{\epsilon}{c_1|\omega \mu_1 - 1|}}{\log \lambda_1}\right)_+
        +(1-p) \left(\frac{\log \frac{\epsilon}{c_2|\omega \mu_2-1|}}{\log \lambda_2}\right)_+,
    \end{equation}
    where $(x)_+ = \max\{0, x\}$.
    Assume $\epsilon< \frac{c_1 c_2 (\mu_2 - \mu_1)}{c_1 \mu_1 + c_2 \mu_2}$.
    If $p \neq p_\epsilon$, (\ref{eq:num_exp}) has the unique minimizer
    \begin{equation}
        \omega_\epsilon =
        \begin{cases}
            \omega_{\epsilon, 1} & \text{if } p>p_\epsilon \\
            \omega_{\epsilon, 2} & \text{if } p<p_\epsilon
        \end{cases}
    \end{equation}
    where
    \begin{equation}
        \omega_{\epsilon, 1} = \frac{1}{\mu_1}-\frac{\epsilon}{c_1 \mu_1}, \quad  \omega_{\epsilon, 2} = \frac{1}{\mu_2}+\frac{\epsilon}{c_2 \mu_2}
    \end{equation}
    and
    \begin{equation}
        p_\epsilon =
        \frac{\log \lambda_1 \log \frac{\epsilon}{c_1(\omega_{\epsilon, 1} \mu_2-1)}}
        {\log \lambda_1 \log \frac{\epsilon}{c_1(\omega_{\epsilon, 1} \mu_2-1)}+\log \lambda_2 \log \frac{\epsilon}{c_2(1-\omega_{\epsilon, 2} \mu_1)}}.
    \end{equation}
    If $p = p_\epsilon$, (\ref{eq:num_exp}) has two different minimizers $\omega_{\epsilon,1}$ and $\omega_{\epsilon,2}$.
    Furthermore, we have
    \begin{equation}
        \lim_{\epsilon \to 0}\omega_\epsilon =
        \begin{cases}
            \frac{1}{\mu_1} & \text{if } p > p_0 \\
            \frac{1}{\mu_2} & \text{if } p < p_0
        \end{cases},
    \end{equation}
    where
    \begin{equation}
        p_0 = \lim_{\epsilon \to 0} p_\epsilon = \frac{\log \lambda_1}{\log \lambda_1 + \log \lambda_2}.
    \end{equation}
\end{proposition}
\begin{proof}[Proof of \cref{lem:write down2}]
    We consider the relaxed version $\L_\epsilon(\tau; \omega) = \min \{ m \in \R_{\geq 0}: \L_m(\tau; \omega) \leq \epsilon^2\}$.
    By solving $\L_n(\tau_i; \omega) = \epsilon^2$ for $n$, we have
    \begin{equation}
        \L_\epsilon(\tau_i; \omega) = \max\left\{0, \frac{\log \frac{\epsilon}{c_i|\omega \mu_i - 1|}}{\log \lambda_i}\right\},
    \end{equation}
    which deduces
    \begin{equation}
        \E_{\tau_i \sim P} [\L_\epsilon(\tau_i; \omega)]
        = p \max\left\{0, \frac{\log \frac{\epsilon}{c_1|\omega \mu_1 - 1|}}{\log \lambda_1}\right\}
        +(1-p) \max\left\{0, \frac{\log \frac{\epsilon}{c_2|\omega \mu_2 - 1|}}{\log \lambda_2}\right\}.
    \end{equation}
    Assuming $\epsilon< \frac{c_1 c_2 (\mu_2 - \mu_1)}{c_1 \mu_1 + c_2 \mu_2}$, this can be written as
    \begin{equation}
        \E_{\tau_i \sim P} [\L_\epsilon(\tau_i; \omega)] =
        \begin{cases}
            p \frac{\log \frac{\epsilon}{c_1(1 - \omega \mu_1)}}{\log \lambda_1}
             & \text{ if } \omega \in\left[\frac{1}{\mu_2}-\frac{\epsilon}{c_2 \mu_2}, \frac{1}{\mu_2}+\frac{\epsilon}{c_2 \mu_2}\right] \\
            (1-p) \frac{\log \frac{\epsilon}{c_2(\omega \mu_2 - 1)}}{\log \lambda_2}
             & \text{ if } \omega \in\left[\frac{1}{\mu_1}-\frac{\epsilon}{c_1 \mu_1}, \frac{1}{\mu_1}+\frac{\epsilon}{c_1 \mu_1}\right] \\
            p \frac{\log \frac{\epsilon}{c_1|\omega \mu_1 - 1|}}{\log \lambda_1}
            + (1-p) \frac{\log \frac{\epsilon}{c_2|\omega \mu_2 - 1|}}{\log \lambda_2}
             & \text{ otherwise. }
        \end{cases}
    \end{equation}
    Note that $\E_{\tau_i \sim P} [\L_\epsilon(\tau_i; \omega)]$ is strictly decreasing for $\omega \in (-\infty, \frac{1}{\mu_2}+\frac{\epsilon}{c_2 \mu_2}]$
        and strictly increasing for $\omega \in [\frac{1}{\mu_1}-\frac{\epsilon}{c_1 \mu_1}, \infty)$.
    Since $\E_{\tau_i \sim P} [\L_\epsilon(\tau_i; \omega)]$ is concave for $\omega \in [\frac{1}{\mu_2}+\frac{\epsilon}{c_2 \mu_2}, \frac{1}{\mu_1}-\frac{\epsilon}{c_1 \mu_1}]$,
    its minimum is attained at $\omega_{\epsilon, 1} = \frac{1}{\mu_1}-\frac{\epsilon}{c \mu_1}$ or $\omega_{\epsilon, 2} = \frac{1}{\mu_2}+\frac{\epsilon}{c_2 \mu_2}$.
    Then, by comparing $\E_{\tau_i \sim P} [\L_\epsilon(\tau_i; \omega_{\epsilon, 2})] = p \L_\epsilon (\tau_1; \omega_{\epsilon, 2})$
    and $\E_{\tau_i \sim P} [\L_\epsilon(\tau_i; \omega_{\epsilon, 1})] = (1-p) \L_\epsilon (\tau_2; \omega_{\epsilon, 1})$,
    we can deduce the result.
\end{proof}

Let us prove \cref{thm:guarantee2} first before \cref{thm:opposite}.
\begin{proof}[Proof of \cref{thm:guarantee2}]
    The limits are shown in \cref{lem:write down2}.

    We now show the second part.
    Note that if $\delta_1 > \delta_2 \geq \epsilon > 0$, then $p_0 < p_\epsilon \leq p_{\delta_2} < p_{\delta_1}$.
    Let $\omega_{\mathrm{max}} = \argmax_{\omega \in [1/\omega_{\epsilon, 2}, 1/\omega_{\epsilon, 1}]} \E_{\tau_i \sim P} [\L_\epsilon(\tau_i; \omega)]$.

    If $p < p_0$, then we have
    $\omega_\epsilon = \omega_{\epsilon, 2} \leq \omega_{\delta_2} = \omega_{\delta_2, 2} < \omega_{\delta_1} = \omega_{\delta_1, 2}$.
    Note that $p < p_0$ and $\delta_1 < \frac{c_1 c_2 (\mu_2 - \mu_1)}{2(c_1 \mu_1 + c_2 \mu_2)}$ guarantee $\omega_{\delta_1} < \omega_{\mathrm{max}}$.
    Since $\E_{\tau_i \sim P} [\L_\epsilon(\tau_i; \omega)]$ is increasing for $\omega \in [\omega_{\epsilon}, \omega_{\mathrm{max}}]$
    and $\omega_\epsilon \leq \omega_{\delta_2} < \omega_{\delta_1} < \omega_\mathrm{max}$,
    we have $\E_{\tau \sim P}[ L_\epsilon(\tau; \omega_{\delta_1}) ] > \E_{\tau \sim P}[ L_\epsilon(\tau; \omega_{\delta_2}) ]$.

    If $p > p_{\delta_1}$, then we have
    $\omega_\epsilon = \omega_{\epsilon, 1} \geq \omega_{\delta_2} = \omega_{\delta_2, 1} > \omega_{\delta_1} = \omega_{\delta_1, 1}$.
    Note that $p > p_0$ and $\delta_1 < \frac{c_1 c_2 (\mu_2 - \mu_1)}{2(c_1 \mu_1 + c_2 \mu_2)}$ guarantee $\omega_{\delta_1} > \omega_{\mathrm{max}}$.
    Since $\E_{\tau_i \sim P} [\L_\epsilon(\tau_i; \omega)]$ is decreasing for $\omega \in [\omega_{\mathrm{max}}, \omega_\epsilon]$
    and $\omega_\epsilon \geq \omega_{\delta_2} > \omega_{\delta_1} > \omega_\mathrm{max}$,
    we have $\E_{\tau \sim P}[ L_\epsilon(\tau; \omega_{\delta_1}) ] > \E_{\tau \sim P}[ L_\epsilon(\tau; \omega_{\delta_2}) ]$.

    We now show the last part.
    Recall that $\omega_m = \frac{p c_1^2 \mu_1 \lambda_1^{2 m}+(1-p) c_2^2 \mu_2 \lambda_2^{2 m}}{p c_1^2 \mu_1^2 \lambda_1^{2 m}+(1-p) c_2^2 \mu_2^2 \lambda_2^{2 m}}$.
    Setting $m = 0$, we have $\omega_0 = \frac{p c_1^2 \mu_1 +(1-p) c_2^2 \mu_2 }{p c_1^2 \mu_1^2 +(1-p) c_2^2 \mu_2^2 }$.
    For $\delta > 0$, we take $c_1$, $c_2$, and $p$ so that $\omega_\epsilon \leq \omega_\delta < \omega_\mathrm{max}<\omega_0 < \omega_m$
    and $\E_{\tau_i \sim P} [\L_\epsilon(\tau_i; \omega_\delta)]$.
    For example, $c_1 = \frac{\mu_2 (\log \lambda_1 + \log \lambda_2)}{p^2(\mu_2 - \mu_1) \log \lambda_2} \delta$,
    $c_2 = \frac{\mu_1(\log \lambda_1 + \log \lambda_2)}{p(\mu_2 - \mu_1) \log \lambda_1} \delta$, and
    $p<p_0$ satisfy the relationship.
    Then, we have
    \begin{align}
        \frac{\E_{\tau_i \sim P} [\L_\epsilon(\tau_i; \omega_m)]}{\E_{\tau_i \sim P} [\L_\epsilon(\tau_i; \omega_\delta)]}
         & = \frac{p \L_\epsilon (\tau_1; \omega_m) + (1-p) \L_\epsilon (\tau_2; \omega_m)} {p \L_\epsilon (\tau_1; \omega_\delta) + (1-p) \L_\epsilon (\tau_2; \omega_\delta)} \\
         & \geq \frac{(1-p) \L_\epsilon (\tau_2; \omega_m)} {p \L_\epsilon (\tau_1; \omega_\delta) + (1-p) \L_\epsilon (\tau_2; \omega_\delta)}                                 \\
         & \geq \frac{(1-p) \L_\epsilon (\tau_2; \omega_0)} {p \L_\epsilon (\tau_1; \omega_\delta) + (1-p) \L_\epsilon (\tau_2; \omega_\delta)}                                 \\
         & =  \frac{(1-p) \log \lambda_1 \log\frac{\epsilon}{c_2(\omega_0\mu_2 - 1)}}
        {p \log \lambda_2 \log\frac{\epsilon}{c_1(1 - \omega_\delta\mu_1)} + (1-p) \log \lambda_1 \log\frac{\epsilon}{c_2(\omega_\delta\mu_2 - 1)}}                             \\
    \end{align}
    Substituting $\omega_0$, $\omega_\delta$, $c_1$, and $c_2$ and taking the limit as $p \to 0$,
    we have $\frac{\E_{\tau_i \sim P} [\L_\epsilon(\tau_i; \omega_m)]}{\E_{\tau_i \sim P} [\L_\epsilon(\tau_i; \omega_\delta)]} \to \infty$.

\end{proof}

\begin{proof}[Proof of \cref{thm:opposite}]
    The limits are shown in \cref{lem:write down}.

    We now show the second part.
    Since $\omega_m$ is increasing in $m$ and $\lim_{m \to \infty} \omega_m = 1/\mu_1$,
    there exists $m_0$ such that $\omega_{\epsilon, 1} < \omega_{m_0}$.
    For any $m_1$ and $m_2$, if $m_0 < m_1 < m_2$, then $ \omega_{m_0} < \omega_{m_1}< \omega_{m_2}$.
    Hence, $\E_{\tau_i \sim P} [\L_\epsilon(\tau_i; \omega_{m_1})] < \E_{\tau_i \sim P} [\L_\epsilon(\tau_i; \omega_{m_2})]$
    because $\E_{\tau_i \sim P} [\L_\epsilon(\tau_i; \omega)]$ is increasing for $\omega \in [\omega_{\epsilon, 1}, 1/\mu_1]$.

    For the last part, the proof is similar to \cref{thm:guarantee2}.
    Take $c_1 = \frac{\mu_2 (\log \lambda_1 + \log \lambda_2)}{p^2(\mu_2 - \mu_1) \log \lambda_2} \epsilon$,
    $c_2 = \frac{\mu_1(\log \lambda_1 + \log \lambda_2)}{p(\mu_2 - \mu_1) \log \lambda_1} \epsilon$, and
    $p<p_0$ and substitute them into $\E_{\tau_i \sim P} [\L_\epsilon(\tau_i; \omega_m)]$.
    Then, we have $\E_{\tau_i \sim P} [\L_\epsilon(\tau_i; \omega_m)] \to \infty$ as $p \to 0$.


\end{proof}



\section{Details of numerical examples}
\label{app:numerical}
\subsection{Details in \texorpdfstring{\cref{sec:counter-example}}{} and \texorpdfstring{\cref{sec:counter-example2}}{}}
\label{app:counter-example}
\paragraph{Task}
In distribution $P_1$, $f_\tau$ is represented by
\begin{equation}
    f_\tau = \sum_{i=1}^{N}c_i \mu_i v_i, \text{ where } c_i \sim \mathcal{N}(0, \left|\frac{N+1-2i}{N-1}\right|).
\end{equation}
In distribution $P_2$, $f_\tau$ is represented by
\begin{equation}
    f_\tau = \sum_{i=1}^{N}c_i \mu_i v_i, \text{ where } c_i \sim \mathcal{N}(0, 1 - \left|\frac{N+1-2i}{N-1}\right|).
\end{equation}
The discretization size is $N=16$ in the experiments.
The number of tasks for training, validation, and test are all $10^3$.

\paragraph{Network architecture and hyper-parameters}
In \cref{sec:counter-example} and \cref{sec:counter-example2}, $\Psi_{\mathrm{nn}}$
is a fully connected neural network with two hidden layers of $15$ neurons.
Its input is discretized $f_\tau \in \R^N$ and the output is $c_i$'s of $\hat u^{(0)}= \sum_{i=1}^N c_i v_i$.
The activation function is $\mathrm{SiLU}$ \cite{Elfwing2018-ae} for hidden layers.
The optimizer is Adam \cite{Kingma2015-ys} with learning rate $0.01$ and $(\beta_1, \beta_2) = (0.999, 0.999)$.
The batch size is $256$.
The model is trained for $2500$ epochs.
During training, if the validation loss does not decrease for $100$ epochs, the learning rate is reduced by a factor of $1/5$.
The presented results are obtained by the best models selected based on the validation loss.


\paragraph{Results}
The detailed results of \cref{fig:performance} are presented in \cref{tab:poisson}.

\begin{table}[h]
    {
        \footnotesize
        \caption{The average number of iterations for solving Poisson equations.
            Boldface indicates best performance for each column.
        }
        \label{tab:poisson}
        \begin{center}
            \begin{tabular}{lll|rrrr|rrrr}
                                     &             &                    & $p=0$             &               &                &                & $p=0.01$                                                         \\
                \hline
                                     &             &                    & $\epsilon$ (test) &               &                &                &               &                                                  \\
                $\Psi$               & $m$ (train) & $\epsilon$ (train) & $10^{-2}$         & $10^{-4}$     & $10^{-6}$      & $10^{-8}$      & $10^{-2}$     & $10^{-4}$      & $10^{-6}$      & $10^{-8}$      \\
                \hline
                $\Psi_{\text{base}}$ & -           & -                  & 28.17             & 92.54         & 158.41         & 224.28         & 30.15         & 96.52          & 164.41         & 232.30         \\
                \hline
                $\Psi_{\text{nn}}$   & 0           & -                  & \textbf{0.00}     & 17.52         & 168.80         & 436.95         & 2.11          & 183.28         & 451.24         & 719.40         \\
                                     & 1           & -                  & \textbf{0.00}     & 20.07         & 148.26         & 403.22         & 1.90          & 178.61         & 446.51         & 714.67         \\
                                     & 5           & -                  & 1.00              & 13.82         & 135.43         & 397.97         & 1.18          & 34.47          & 286.43         & 554.55         \\
                                     & 25          & -                  & 6.26              & 16.31         & 37.71          & 172.02         & 7.11          & 20.22          & 121.24         & 379.51         \\
                                     & 125         & -                  & 23.42             & 83.28         & 148.73         & 214.59         & 20.96         & 80.44          & 151.39         & 353.35         \\
                \hline
                                     & -           & $10^{-2}$          & 0.01              & 40.31         & 223.15         & 491.30         & \textbf{0.25} & 185.93         & 453.84         & 721.99         \\
                                     & -           & $10^{-4}$          & 3.93              & \textbf{9.76} & 48.61          & 178.15         & 5.34          & \textbf{13.69} & 176.77         & 444.71         \\
                                     & -           & $10^{-6}$          & 6.20              & 15.79         & \textbf{32.00} & \textbf{83.43} & 7.94          & 19.79          & \textbf{36.11} & 166.30         \\
                                     & -           & $10^{-8}$          & 11.12             & 32.05         & 57.62          & 86.70          & 13.39         & 37.44          & 66.16          & \textbf{97.97} \\
            \end{tabular}
        \end{center}
    }
\end{table}

\subsection{Details in \texorpdfstring{\cref{sec:relaxation factors}}{}}
\label{app:relax toy}

\paragraph{Task}
We prepare $10,000$ sets of $c_1,c_2,c_3$ and use $2,500$ for training, $2,500$ for validation, and $5,000$ for test.
Since the solution of the Robertson equation has a quick initial transient followed by a smooth variation \cite{Hairer2010-uw},
we set the step size $h_n \ (n=1, 2, \ldots, 100)$ so that the evaluation points are located log-uniformly in $[10^{-6}, 10^3]$.
Thus, each set of $c_1,c_2,c_3$ is associated with $100$ data points.

\paragraph{Solver}
The iteration rule of the Newton-SOR method is
\begin{align}
    J_n(y_{n+1}^{(k)}) & = D_k - L_k - U_k                                          \\
    \label{eq:Newton-SOR}
    y_{n+1}^{(k+1)}    & = y_{n+1}^{(k)} - r (D_k - r L_k)^{-1} g_n(y_{n+1}^{(k)}),
\end{align}
where $J_n$ is the Jacobian of $g_n$, its decomposition $D_k$, $L_k$, $U_k$ are diagonal, strictly lower triangular, and strictly upper triangular matrices respectively,
and $r \in [1, 2]$ is the relaxation factor of the SOR method.
It iterates until the residual of the approximate solution $\norm{g_n(y_{n+1}^{(k)})}$ reaches a given error tolerance $\epsilon$.
Note that we choose to investigate the Newton-SOR method, because it is a fast and scalable method and applicable to larger problems.

\paragraph{Network architecture and hyper-parameters}
In \cref{sec:relaxation factors}, $\Psi_{\text{ini}}$, $\Psi_{\text{relax}}$, and $\Psi_{\text{both}}$ are fully-connected neural networks
that take $c_1, c_2, c_3, h_n, y_n$ as an input.
They have two hidden layers with 1024 neurons, and ReLU is used as the activation function except for the last layer.
The difference among them is only the last layer.
The last layer of $\Psi_{\text{ini}}$ modifies the previous timestep solution $y_n$ for a better initial guess $y_\tau \in \R^3$.
Since the scale of each element of $y_n$ is quite different, the modification is conducted in log scale, i.e.
\begin{equation}
    y_\tau = \exp (\log (y_n + \mathrm{tanh}(W_{\text{ini}}x))),
\end{equation}
where $x$ is the input of the last layer and $W_{\text{ini}}$ are its weights.
The last layer of $\Psi_{\text{relax}}$ is designed to output the relaxation factor $r_\tau \in [1, 2]$:
\begin{equation}
    r_\tau = \mathrm{sigmoid}(W_{\text{relax}}x) + 1,
\end{equation}
where $x$ is the input of the last layer and $W_{\text{relax}}$ are its weights.
The last layer of $\Psi_{\text{both}}$ is their combination.

The meta-solvers are trained for 200 epochs by Adam with batchsize $16384$.
The initial learning rate is $2.0 \cdot 10^{-5}$, and it is reduced at the $100$th epoch and $150$th epoch by $1/5$.
For $\Psi_{\text{relax}}$ and $\Psi_{\text{both}}$, the bias in the last layer corresponding to the relaxation factor is set to $-1$ at the initialization
for preventing unstable behavior in the beginning of the training.

\subsection{Details in \texorpdfstring{\cref{sec:preconditioning}}{}}

\paragraph{Task}
To set up problem instances, open-source computing library FEniCS \cite{Alnaes2015-nz} is used.
We sampled tasks $1,000$ for training, $1,000$ for validation, and $1,000$ for test.

\paragraph{Solver}
Our implementation of the ICCG method is based on \cite{Huang2009-jv}.

\paragraph{Network architecture and hyper-parameters}
In \cref{sec:preconditioning}, $\Psi_{\text{nn}}$ is a fully-connected neural network with two hidden layers of $128$ neurons.
Its input is $\{W_\tau, \lambda_\tau, \mu_\tau, \rho_\tau, \norm{A_\tau}_1, \norm{A_\tau}_\infty, \norm{A_\tau}_F \}$
and output is diagonal shift $\alpha_\tau$.
The activation function is ReLU except for the last layer.
At the last layer, the sigmoid function is used as the activation function.
The hyper-parameter $\gamma$ in the loss functions is $\gamma=500$.

The meta-solver is trained for $200$ epochs by Adam with batchsize $128$.
The initial learning rate is $0.001$ and it is reduced at the $100$th epoch and $150$th epoch by $1/5$.
The presented results are obtained by the best models selected based on the validation loss.

\label{app:preconditioning}



\end{document}